\documentclass[reqno]{amsart}
\usepackage{tikz}
\usetikzlibrary{patterns} 
\usepackage{pgfplots}
\usepackage{amsfonts,amsmath,amsthm,amssymb,mathrsfs}
\usepackage{color}
\usepackage{amsmath,amssymb,amsfonts,bm}
\usepackage{graphicx}
\usepackage{newunicodechar}
\usepackage{xcolor}
\usepackage{geometry}
\numberwithin{equation}{section}
\usepackage[pdfstartview=FitH,colorlinks=true]{hyperref}
\hypersetup{urlcolor=red, linkcolor=blue,citecolor=red}
\allowdisplaybreaks

\theoremstyle{plain}

\newtheorem{thm}{Theorem}[section]

\newtheorem{prop}[thm]{Proposition}
\newtheorem{cor}[thm]{Corollary}
\newtheorem{lemma}[thm]{Lemma}
\newtheorem{remark}[thm]{Remark}

\begin{document}
\title[Landau equation]
{Regularizing effect of the spatially homogeneous Landau equation with soft potential}

\author[ ]
{ Xiaodong Cao \& Chaojiang Xu \& Yan Xu}

\address{Xiao-Dong Cao, Chao-Jiang Xu
\newline\indent
School of Mathematics and Key Laboratory of Mathematical MIIT,
\newline\indent
Nanjing University of Aeronautics and Astronautics, Nanjing 210016, China
\newline\indent
Yan Xu
\newline\indent
Department of Mathematical Sciences, Tsinghua University, Beijing 100084, China
}
\email{caoxiaodong@nuaa.edu.cn; xuchaojiang@nuaa.edu.cn; xu-y@mail.tsinghua.edu.cn}

\date{\today}



\begin{abstract}
This paper investigates the Cauchy problem of the spatially homogeneous Landau equation with soft potential under the perturbation framework to global equilibrium. We prove that the solution to the Cauchy problem exhibits analyticity in the time variable and the Gelfand-Shilov regularizing effect in the velocity variables.
\end{abstract}

\maketitle

\section{Introduction}
\par The Cauchy problem for the spatially homogeneous Landau equation is given by
\begin{equation}\label{1-1}
\left\{
\begin{aligned}
 &\partial_t F=Q(F, F),\\
 &F|_{t=0}=F_0,
\end{aligned}
\right.
\end{equation}
where $F=F(t, v)\ge0$ represents the density distribution function at time $t\ge0$, with the velocity variable $v\in\mathbb R^3$. The Landau bilinear collision operator is defined as follows: 
\begin{equation}\label{Q}
    Q(G, F)(v)=\sum_{j, k=1}^3\partial_j\bigg(\int_{\mathbb R^3}a_{jk}(v-v_*)[G(v_*)\partial_kF(v)-\partial_kG(v_*)F(v)]dv_*\bigg),
\end{equation}
where the non-negative symmetric matrix $\left(a_{jk}\right)$ is given by
\begin{equation}\label{matrix A}
   a_{jk}(v)=(\delta_{jk}|v|^2-v_jv_k)|v|^\gamma,\quad \gamma\ge-3.
\end{equation}
The parameter $\gamma$ leads to the classification of hard potential if $\gamma>0$, Maxwellian molecules if $\gamma=0$, soft potential if $-3<\gamma<0$ and Coulombian potential if $\gamma=-3$.

The Landau equation is one of the fundamental kinetic equations obtained as a limit of the Boltzmann equation when grazing collisions prevail~\cite{V-1}. Extensive research has been conducted on the Landau equation with hard potential. Desvillettes-Villani~\cite{D-2} obtained the smoothness of solutions to the spatially homogeneous Landau equation.

For the hard potential case, the analytic smoothing effect for the spatially homogeneous Landau equation was established in~\cite{C-1}, assuming an initial datum with finite mass and energy. Under the perturbation setting to the normalized global Maxwellian, the analytic smoothing effect to solution has been studied in~\cite{L-4},  Gevrey regularity was treated in~\cite{C-2, CLX}, analytic Gelfand-Shilov smoothing effect was proven in~\cite{LX-5}. For the Maxwellian molecules case, the existence, uniqueness and smoothness of solutions to the spatially homogeneous equation have been studied in~\cite{V-2}, again under the assumption of an initial datum with finite mass and energy. Additionally, the analytic and Gelfand-Shilov smoothing effects were addressed in~\cite{L-2, M-1, MX}.

We are now interested in the spatially homogeneous Landau equation with soft potentials. In this work, we focus on studying the linearization of the Landau equation \eqref{1-1} around the Maxwellian distribution
$\mu(v)=(2\pi)^{-\frac32}e^{-\frac{|v|^2}{2}}. $
Considering the fluctuation of the density distribution function
\begin{equation*}
    F(t, v)=\mu(v)+\sqrt\mu(v)f(t, v),
\end{equation*}
and noting that $Q(\mu, \mu)=0$, the Cauchy problem \eqref{1-1} can be reformulated as
\begin{equation}\label{1-2}
\left\{
\begin{aligned}
 &\partial_tf+\mathcal Lf=\Gamma(f, f),\\
 &f|_{t=0}=f_0,
\end{aligned}
\right.
\end{equation}
where $F_0=\mu+\sqrt\mu f_0$, and
\begin{equation*}
    \Gamma(f, f)=\mu^{-\frac12}Q(\sqrt\mu f, \sqrt\mu f),
\end{equation*}
with the linear operator $\mathcal{L}$ defined as
\begin{equation}\label{linear operator}
\mathcal{L}=\mathcal{L}_1+\mathcal{L}_2,\ \ \ \mbox{with}\ \ \    \mathcal L_1f=-\Gamma(\sqrt\mu, f), \quad \mathcal L_2f=-\Gamma(f, \sqrt\mu).
\end{equation}
The diffusion part $\mathcal L_1$ is expressed as
\begin{equation}\label{1-3}
    \mathcal L_1f=-\nabla_v\cdot[A(v)\nabla_vf]+\left(A(v)\frac{v}{2}\cdot\frac{v}{2}\right)f-\nabla_v\cdot\left[A(v)\frac{v}{2}\right]f,
\end{equation}
where $A(v)=(\bar a_{jk})_{1\le j, k\le3}$ is a symmetric matrix, and
\begin{align*}
    \bar a_{jk}=a_{jk}*\mu=\int_{\mathbb R^3}\left(\delta_{jk}|v-v'|^2-(v_j-v'_j)(v_k-v'_k)\right)|v-v'|^\gamma\mu(v')dv'.
\end{align*}
Thus
$$
 \mathcal L_1f \ \sim\ \langle v\rangle^{\gamma}(-\Delta_{v}+\frac{|v|^{2}}{4}),
$$
and this operator is degenerate at $|v|=+\infty$ for the soft potential case $\gamma<0$.

For soft potentials, existence and uniqueness results can be found in~\cite{V-1, W}. Regarding smoothing properties, it was shown in~\cite{L-5} that solutions to the linear Landau equation with soft potentials exhibit analytic smoothness in both the velocity and time variables. The Gelfand-Shilov smoothing effect for moderately soft potentials $-2<\gamma<0$, was examined in~\cite{L-4}. Additionally, various regularity results in a near-equilibrium setting have been discussed in~\cite{C-4, C-5, S-1}.

We mention that our discussion is based on the following function spaces. Let $\Omega$ be an open subset of $\mathbb R^{n}$. The analytical space $\mathcal A(\Omega)$, consists of smooth functions that satisfy the following condition: $f\in \mathcal A(\Omega)$ if there exists a constant $C$ such that for all $\alpha\in\mathbb N^n$
\begin{align*}
    \left\|\partial^\alpha f\right\|_{L^2(\Omega)}\le C^{|\alpha|+1}\alpha!.
\end{align*}
The Gelfand-Shilov spaces $S^{\sigma}_{\nu}(\mathbb R^{n})$, with $\sigma, \nu>0, \sigma+\nu\ge1$, consist of smooth functions that satisfy the following condition: $f\in S^{\sigma}_{\nu}(\mathbb R^{n})$ if there exists a constant $C$ such that for all multi-indices $\alpha, \beta\in\mathbb N^n$,
\begin{align*}
    \left\|v^{\beta}\partial^\alpha f\right\|_{L^2(\mathbb R^{n})}\le C^{|\alpha|+|\beta|+1}(\alpha!)^{\sigma}(\beta!)^{\nu}.
\end{align*}

Before presenting the main result of the paper, we first introduce some notation. We denote the harmonic oscillators operator by $\mathcal H=-\Delta_{v}+\frac{|v|^{2}}{4}$ and the associated gradient by
\begin{align*}
    \nabla_{\mathcal H_{\pm}}=(A_{\pm, 1}, A_{\pm, 2}, A_{\pm, 3}),
\end{align*}
where
\begin{align*}
    A_{\pm, k}=\frac12v_{k}\mp\partial_{v_{k}}, (1\le k\le3), \quad A_{\pm}^{\alpha}=A_{\pm, 1}^{\alpha_{1}}A_{\pm, 2}^{\alpha_{2}}A_{\pm, 3}^{\alpha_{3}}, \ (\alpha\in\mathbb N^{3}).
\end{align*}
We also define, for some constants $c_0>0$ and $0<b\le 2$,  
$$
\omega_t(v)=e^{\frac{c_0}{1+t}\langle v\rangle^b},\ \ \ t\ge 0, \ v\in \mathbb{R}^3,
$$
with $\langle v\rangle=(1+|v|^2)^{\frac 12}$. 
Now, the main result of this paper is stated as follows.

\begin{thm}\label{thm}
For the soft potential $-3<\gamma<0$, assume that there exist  positive constants $c_{0}$ and $0<b\le 2$  such that for any $f_{0}\in L^{2}(\mathbb R^{3})$ satisfying $\|\omega_0 f_{0}\|_{L^{2}(\mathbb R^{3})}$ small enough, then the Cauchy problem \eqref{1-2} admits a unique solution
$\omega_{t} f\in \mathcal A(]0, \infty[; S^{\sigma}_{\sigma}(\mathbb R^{3}))$, where $\sigma=\max\{1, \frac{b-\gamma}{2b}\}$. Furthermore, for any  $T>0$, there exists a constant $C_{T}>0$ which depends on $\gamma, \ b, \ c_{0}$ and $T$, such that
\begin{align}\label{main result}
     \left\|\omega_t \nabla^{m}_{\mathcal H_{+}}f(t)\right\|_{L^{2}(\mathbb R^{3})}\le \left(\frac{C_T^{m+1} m!}{t^{m}}\right)^{\sigma}, \quad \forall \ m\in\mathbb N, \ \forall \ 0<t\le T,
\end{align}
and
\begin{equation*}
    \|\partial^m_tf(t)\|_{L^2(\mathbb R^3)}\le\frac{C_T^{m+1}}{t^m}m!,\quad \forall \ m\in\mathbb N, \ \forall \ 0<t\le T.
\end{equation*}
\end{thm}

\begin{remark}
The results of reference ~\cite{LX} indicate that the solution of \eqref{1-2} belongs to $S^{1}_{1}(\mathbb R^{3})$ when $-1\le\gamma<0$, and belongs to $S^{\frac{1}{\gamma + 2}}_{\frac{1}{\gamma + 2}}(\mathbb R^{3})$ when $-2<\gamma<-1$. To overcome the difficulty of degeneracy, we introduce an exponential type weight $\omega_{t}(v)=e^{\frac{c_{0}}{1+t}\langle v\rangle^b}$ where the parameter $b$ affects the analyticity index. Specifically, it implies that: the solution belongs to $\mathcal A(]0, \infty[; S^{1}_{1}(\mathbb R^{3}))$ for $-b\le\gamma<0$, and further proven that it belongs to $\mathcal A(]0, \infty[; S^{\frac{b-\gamma}{2b}}_{\frac{b-\gamma}{2b}}(\mathbb R^{3}))$ for $-3<\gamma<-b$. So that we can obtain the Gelfand-Shilov smoothing effect for all soft potential.
\end{remark}

\begin{remark}
In~\cite{HJL}, they consider the spatially inhomogeneous Boltzmann equation without angular cutoff for soft potentials and obtained the global Gevrey regularity with the Gevrey index $\max\{\frac{2-\gamma}{4s}, 1\}$, for initial data with some exponential weight(where the parameter $b=2$). 
\end{remark}

\begin{remark}
In our main theorem \ref{thm}, by Proposition 5.2 of \cite{LX-5} and Theorem 2.1 of \cite{GPR}, we can obtain
$$
 \left\|\omega_t v^\beta \partial^\alpha_v f(t)\right\|_{L^{2}(\mathbb R^{3})}\le \left(\frac{C_T^{\alpha|+|\beta|+1} \beta! \alpha! }{t^{|\alpha|+|\beta|}}\right)^{\sigma}, \quad \forall \ \alpha, \beta\in\mathbb N^3, \ \forall \ 0<t\le T.
$$
Then we have
\begin{align*}
 e^{\frac {c_0}{1+T}\langle v\rangle^b  +c_1t \langle v \rangle^{\frac{1}{\sigma}}} f \in L^2(\mathbb{R}^3), \quad 0 < t \le T.
\end{align*}
(1) For the case of $-b \le \gamma \le 0$, we have proved that
$$
e^{\frac {c_0}{1+T}\langle v \rangle^b + c_1 t \langle v \rangle} f \in L^2(\mathbb{R}^3), \quad 0 < t \le T.
$$
Thus if $b \ge 1$, we have not improved the decreasing rate, and we obtain 
$$
e^{\frac {c_0}{1+T}\langle v \rangle^b} f \in L^2(\mathbb{R}^3), \quad 0 < t \le T,
$$
on the other hand, if $b < 1$, we improve the decreasing rate, i. e.
$$
 e^{c_1 t \langle v \rangle} f \in L^2(\mathbb{R}^3),  \quad 0 < t \le T.
$$
(2) For the case of $-3 < \gamma < -b$, we have proved that
$$
e^{\frac {c_0}{1+T}\langle v \rangle^b + c_1 t \langle v \rangle^{\frac{2b}{b - \gamma}}} f \in L^2(\mathbb{R}^3),  \quad 0 < t \le T.
$$
Thus if $0 < \frac{2}{b - \gamma} \le 1$, the solution has the same decreasing rate as the initial data, i. e.
$$
e^{\frac {c_0}{1+T} \langle v \rangle^b} f \in L^2(\mathbb{R}^3), \quad 0 < t \le T.
$$
On the other hand, once $\frac{2}{b - \gamma} > 1$, then 
we improve the decreasing rate, i. e.
$$
e^{c_1t \langle v \rangle^{\frac{2b}{b - \gamma}}} f \in L^2(\mathbb{R}^3), \quad 0 < t \le T.
$$

The above analysis is easier to understand with the following figure: only in the red area (2),  we improve the decreasing rate with respect to the initial date.

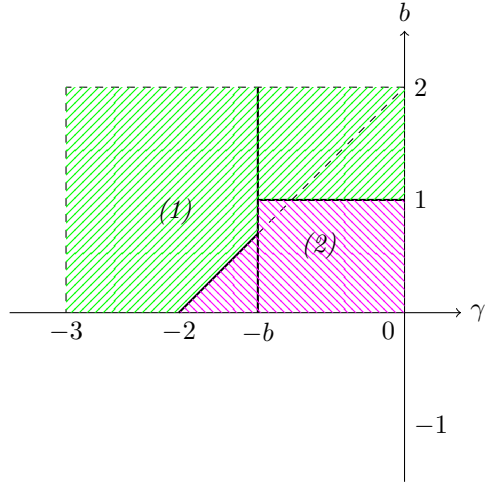
\begin{center}
\begin{tikzpicture}[scale=1.5]

\draw[->] (-3.5,0) -- (0.5,0) node[right] {$\gamma$};
\draw[->] (0,-1.5) -- (0,2.5) node[above] {$b$};

\node[below] at (-3,0) {$-3$};
\node[below] at (-2,0) {$-2$};
\node[below left] at (0,0) {$0$};
\node[right] at (0,2) {$2$};
\node[right] at (0,1) {$1$};
\node[right] at (0,-1) {$-1$};

\node[below] at (-1.3, 0) {$-b$};

\draw[dashed] (0,2) -- (-3,2);
\draw[dashed] (-3,2) -- (-3,0);

\draw[thick] (-1.3,1) -- (0,1);
\draw[thick] (-1.3,0) -- (-1.3,2);

\draw[dashed] (-1.3,0.7) -- (0,2) node[right] {};
\draw[thick] (-2,0) -- (-1.3,0.7) node[right] {};

\fill[pattern=north east lines, pattern color=green] (-3,2) -- (-3,0) -- (-2,0) -- (-1.3,0.7) -- (-1.3,2) -- cycle; 
\fill[pattern=north east lines, pattern color=green] (-1.3,2) -- (-1.3,1) -- (0,1) -- (0,2) -- cycle;

\fill[pattern=north west lines, pattern color=magenta] (-2,0) -- (0,0) -- (0,1) -- (-1.3,1) -- (-1.3,0.7) -- cycle; 

\node[above left] at (-1.8,0.7) {(1)};
\node[below right] at (-1,0.8) {(2)};

\end{tikzpicture}
\end{center}

\end{remark}

\section{Schema of proof for the main theorem}\label{sch}

We begin by recalling some preliminary results regarding Landau operators. For simplicity, with $r\in\mathbb R$, we denote the weighted Lebesgue spaces
\begin{align*}
    \|\langle\cdot\rangle^r f\|_{L^p(\mathbb R^3)}=\|f\|_{p, r},\quad 1\le p\le\infty,
\end{align*}
where the notation $\langle v\rangle=(1+|v|^2)^{\frac12}$. For the matrix $A$ defined in \eqref{1-3}, denote
\begin{align*}
    \|f\|^2_{A}=\sum_{j, k=1}^3\int\left(\bar a_{jk}\partial_jf\partial_kf+\frac14\bar a_{jk}v_jv_kf^2\right)dv=\frac12\sum_{j, k=1}^3\int\bar a_{jk}\left(A_{-, j}fA_{-, k}f+A_{+, j}fA_{+, k}f\right)dv.
\end{align*}
From corollary 1 in~\cite{G-1} and Proposition 2.3 in~\cite{LX-5}, for $\gamma>-3$, there exists a constant $C_1>0$ such that
\begin{equation}\label{2-1}
    \|f\|^2_{A}\ge C_1\left(\|\mathbf P_v\nabla f\|^2_{2, \frac{\gamma}{2}}+\|(\mathbf I-\mathbf P_v)\nabla f\|^2_{2, 1+\frac{\gamma}{2}}+\|f\|^2_{2, 1+\frac{\gamma}{2}}\right),
\end{equation}
and
\begin{equation}\label{2-1-1}
    \|f\|^2_{A}\ge C_1\left(\left\|\mathbf P_v\nabla_{\mathcal H_{\pm}}  f\right\|^2_{2, \frac{\gamma}{2}}+\left\|(\mathbf I-\mathbf P_v)\nabla_{\mathcal H_{\pm}} f\right\|^2_{2, 1+\frac{\gamma}{2}}\right).
\end{equation}
For any vector-valued function $G(v)=(G_1, G_2, G_3)$, defining the projection to the vector $v\in\mathbb R^3$ as
\begin{equation*}
    (\mathbf P_vG)_j=\sum_{k=1}^3G_kv_k\frac{v_j}{|v|^2},\quad 1\le j\le3.
\end{equation*}
Since $\nabla f=\mathbf P_v\nabla f+(\mathbf I-\mathbf P_v)\nabla f$,
combining the inequality \eqref{2-1}, we have
\begin{equation}\label{2-2}
    \|f\|_{A}\ge C_1\left(\|\nabla f\|_{2,\frac{\gamma}{2}}+\|f\|_{2, 1+\frac{\gamma}{2}}\right),
\end{equation}
and
\begin{equation}\label{2-2-1}
    \|f\|_{A}\ge C_1\left\|\nabla_{\mathcal H_{\pm}} f\right\|_{2,\frac{\gamma}{2}}.
\end{equation}

We begin with a singular integration. Review that, for any $\gamma>-3$ and $\delta>0$, we have
\begin{equation}\label{singular integration}
    \int_{\mathbb R^3}|v-v'|^{\gamma}e^{-\delta|v'|^2}dv'\le C_{\gamma, \delta}\langle v\rangle^{\gamma}.
\end{equation}
For later use, we need the following results for the linear Landau operator, which have been proved in~\cite{XX}.
\begin{lemma}~\cite{XX}\label{lemma2.1}
    Let $-3<\gamma<0$, then for any $0<\epsilon_1<1$, there exists a constant $C_{\epsilon_1}>0$ such that for any suitable function $f$
    \begin{equation*}
       (1-\epsilon_1)\|f\|^2_{A}\le (\mathcal L_1f,  f)_{L^2}+C_{\epsilon_1}\|f\|^2_{2, \frac{\gamma}{2}}.
    \end{equation*}
\end{lemma}

Now, we construct the trilinear estimate of the nonlinear Landau operator. 

\begin{prop}\label{trilinear}
Let $-3<\gamma<0$, then there exists a constant $C_{2}>0$, only depends on $\gamma$, such that for any suitable functions $F, G$ and $H$, we have
\begin{align*}
     \left|\left(\Gamma(F, G), H\right)_{L^{2}(\mathbb R^{3})}\right|\le C_{2}\left\|F\right\|_{2, \frac\gamma2}\left\|G\right\|_{A}\left\|H\right\|_{A}.
\end{align*}
\end{prop}
The proof of this Proposition is similar to that in~\cite{L-4}, we just need to pay more attention to the parameter $\gamma<0$.

\bigskip
Using the equation \eqref{1-2}, we can obtain the following $L^{2}$-estimate
\begin{align}\label{L2-estimate}
\begin{split}
    &\frac12\frac{d}{dt}\left\|t^{m\sigma}\nabla_{\mathcal H_{+}}^{m}f(t)\right\|^{2}_{L^{2}(\mathbb R^{3})}+\left(t^{m\sigma}\nabla_{\mathcal H_{+}}^{m}\mathcal Lf, t^{m\sigma}\nabla_{\mathcal H_{+}}^{m}f\right)_{L^{2}(\mathbb R^{3})}\\
    &=mt^{2m\sigma-1}\left\|\nabla_{\mathcal H_{+}}^{m}f(t)\right\|^{2}_{L^{2}(\mathbb R^{3})}+\left(t^{m\sigma}\nabla_{\mathcal H_{+}}^{m}\Gamma(f, f), t^{m\sigma}\nabla_{\mathcal H_{+}}^{m}f\right)_{L^{2}(\mathbb R^{3})}.
\end{split}
\end{align}
Using Lemma \ref{lemma2.1},  we have
$$
  \left(t^{m\sigma}\mathcal L_{1}\nabla_{\mathcal H_{+}}^{m}f, t^{m\sigma}\nabla_{\mathcal H_{+}}^{m}f\right)_{L^{2}(\mathbb R^{3})}\ge\frac 12 \left\|t^{m\sigma}\nabla^{m}_{\mathcal H_{+}}f\right\|^{2}_{A}-C_{3}\left\|t^{m\sigma}\nabla^{m}_{\mathcal H_{+}} f\right\|^{2}_{2, \frac\gamma2}
 $$
then we can obtain
\begin{align*}
    &\frac12\frac{d}{dt}\left\|t^{m\sigma}\nabla_{\mathcal H_{+}}^{m}f(t)\right\|^{2}_{L^{2}(\mathbb R^{3})}+\frac 12\left\|t^{m\sigma}\nabla^{m}_{\mathcal H_{+}}f\right\|^{2}_{A}\\
    &\le\left|(m\sigma)t^{2m\sigma-1}\left\|\nabla_{\mathcal H_{+}}^{m}f(t)\right\|^{2}_{L^{2}(\mathbb R^{3})}\right|+C_{3}\left\|t^{m\sigma}\nabla^{m}_{\mathcal H_{+}} f\right\|^{2}_{2, \frac\gamma2}\\
    &\quad+\left(t^{m\sigma}\nabla_{\mathcal H_{+}}^{m}\Gamma(f, f), t^{m\sigma}\nabla_{\mathcal H_{+}}^{m}f\right)_{L^{2}(\mathbb R^{3})}+\left|\left(t^{m\sigma}\nabla_{\mathcal H_{+}}^{m}\mathcal L_{2}f, t^{m\sigma}\nabla_{\mathcal H_{+}}^{m}f\right)_{L^{2}(\mathbb R^{3})}\right|\\
    &\quad+\left|\left(t^{m\sigma}[\nabla_{\mathcal H_{+}}^{m},\ \mathcal L_{1}]f, t^{m\sigma}\nabla_{\mathcal H_{+}}^{m}f\right)_{L^{2}(\mathbb R^{3})}\right|.
\end{align*}
The second term in the right hand side $\left\|t^{m\sigma}\nabla^{m}_{\mathcal H_{+}} f\right\|^{2}_{2, \frac\gamma2}$ can be treated by using Gronwall's inequality. The third-fifth term can be treated as the hard potential case using induction.

However the first term $(m\sigma) t^{2m\sigma-1}\left\|\nabla_{\mathcal H_{+}}^{m}f(t)\right\|^{2}_{L^{2}(\mathbb R^{3})}$ cannot be bounded through induction. Since
\begin{align*}
&(m\sigma) t^{2m\sigma-1}\left\|\nabla_{\mathcal H_{+}}^{m}f(t)\right\|^{2}_{L^{2}(\mathbb R^{3})}=
(m\sigma) t^{2m\sigma-1}\left( \nabla_{\mathcal H_{+}}^{m}f(t), \ \nabla_{\mathcal H_{+}}^{m}f(t)\right)_{L^{2}(\mathbb R^{3})}\\
&\le (m\sigma)^2 t^{2(m\sigma-1)}\left\|\nabla_{\mathcal H_{+}} \nabla_{\mathcal H_{+}}^{m-1}f(t)\right\|^{2}_{L^{2}(\mathbb R^{3})}+ t^{2m\sigma}\left\|\nabla_{\mathcal H_{+}}^{m}f(t)\right\|^{2}_{L^{2}(\mathbb R^{3})},
\end{align*}
and for the soft potential case $\gamma<0$, due to \eqref{2-2-1},
the term $\left\|\nabla_{\mathcal H_{+}} \nabla_{\mathcal H_{+}}^{m-1}f(t)\right\|^{2}_{L^{2}(\mathbb R^{3})}$ cannot be controlled by
$\left\|\nabla_{\mathcal H_{+}}^{m-1}f(t)\right\|^{2}_{A}$.

On the other hand, by using the interpolation,  for $b>0,\  \gamma<0$ and $\sigma=\max\{1, \frac{b-\gamma}{2b}\}$, we have that for
$0<t\le T$,
\begin{align*}
    &(m\sigma) t^{2m\sigma-1}\left\| \nabla_{\mathcal H_{+}}^{m} f(t)\right\|^{2}_{L^{2}(\mathbb R^{3})}\\
     &= (m\sigma)t^{2m\sigma-1}\int_{\mathbb R^{3}}\left|\langle v\rangle^{\frac b2} \nabla_{\mathcal H_{+}}^{m} f(t, v)\right|^{2\cdot\frac{-\gamma}{b-\gamma}}\left|\langle v\rangle^{\frac\gamma2} \nabla_{\mathcal H_{+}}^{m} f(t, v)\right|^{2\cdot\frac{b}{b-\gamma}}dv\\
     &\le  (m\sigma) t^{\frac{2b\sigma}{b-\gamma}-1}\left\|\langle\cdot\rangle^{\frac b2} t^{m\sigma}\nabla_{\mathcal H_{+}}^{m} f(t)\right\|^{\frac{-2\gamma}{b-\gamma}}_{L^{2}(\mathbb R^{3})}\left(\frac{1}{C_{1}}\left\|t^{(m-1)\sigma}\nabla_{\mathcal H_{+}}^{m-1} f(t)\right\|_{A}\right)^{\frac{2b}{b-\gamma}}\\
     &\le \epsilon \left\|\langle\cdot\rangle^{\frac b2} t^{m\sigma}\nabla_{\mathcal H_{+}}^{m} f(t)\right\|^{2}_{L^{2}(\mathbb R^{3})}+C_\epsilon \left(\frac{\tilde{T}_0 m\sigma}{C_1}\right)^{\frac{b}{b-\gamma}}
\left\|t^{(m-1)\sigma}\nabla_{\mathcal H_{+}}^{m-1} f(t)\right\|_{A}^{2} \\
&\le\epsilon \left\|\langle\cdot\rangle^{\frac b2} t^{m\sigma}\nabla_{\mathcal H_{+}}^{m} f(t)\right\|^{2}_{L^{2}(\mathbb R^{3})}+\tilde{C}_\epsilon (m-1)^{2\sigma}
\left\|t^{(m-1)\sigma}\nabla_{\mathcal H_{+}}^{m-1} f(t)\right\|_{A}^{2} ,
\end{align*}
here we use  $\frac{2b\sigma}{b-\gamma}-1\ge 0$, and $\tilde{T}_0=T^{\frac{2b\sigma}{b-\gamma}-1}$. Now, the second term on the right-hand side can be controlled by induction, but the first term on the right-hand side is with a positive weight $\langle v \rangle^{\frac b2}$.
So we need to use the following exponential-type weighted functions
\begin{align*}
\omega_{t}(v) = e^{\frac{c_{0}}{1+t}\langle v\rangle^{b}},
\end{align*}
with $c_0>0$, $0<b\le2$, then with the equation
$$
\partial_{t}\left(\omega_{t} f\right)+\frac{c_{0}}{(1+t)^{2}}{\langle v\rangle^{b}}(\omega_{t} f)+\omega_{t}\mathcal Lf=\omega_{t}\Gamma(f, f),
$$
we can establish the following $L^{2}$-estimate
\begin{align*}
     &\frac12\frac{d}{dt}\left\|\omega_{t} t^{m\sigma}\nabla_{\mathcal H_{+}}^{m} f(t)\right\|^2_{L^2(\mathbb R^3)}+\frac{c_0}{(1 + t)^2}\left\|\langle\cdot\rangle^{\frac b2}\omega_{t} t^{m\sigma}\nabla_{\mathcal H_{+}}^{m} f(t)\right\|^{2}_{L^{2}(\mathbb R^{3})}+\left(\omega_{t} t^{m\sigma}\nabla_{\mathcal H_{+}}^{m} \mathcal Lf, \ \omega_{t} t^{m\sigma}\nabla_{\mathcal H_{+}}^{m} f\right)_{L^{2}(\mathbb R^{3})}\\
     &=(m\sigma) t^{2m\sigma-1}\left\|\omega_{t} \nabla_{\mathcal H_{+}}^{m} f(t)\right\|^{2}_{L^{2}(\mathbb R^{3})}+\left(\omega_{t} t^{m\sigma}\nabla_{\mathcal H_{+}}^{m} \Gamma(f, f), \ \omega_{t} t^{m\sigma}\nabla_{\mathcal H_{+}}^{m} f\right)_{L^{2}(\mathbb R^{3})},
\end{align*}
The primary distinction from \eqref{L2-estimate} lies in the left-hand side above, where we gain a term 
$$\left\|\langle\cdot\rangle^{\frac b2}\omega_{t} t^{m\sigma}\nabla_{\mathcal H_{+}}^{m} f(t)\right\|^{2}_{L^{2}(\mathbb R^{3})}.$$
Consequently, the term $(m\sigma) t^{2m\sigma-1}\left\|\omega_{t} \nabla_{\mathcal H_{+}}^{m} f(t)\right\|^{2}_{L^{2}(\mathbb R^{3})}$ can be bounded by induction.

Next, we will focus on the estimation of commutators of Landau operators with weight $\omega_{t}$ and $\nabla_{\mathcal H_{+}}^{m}$,as stated in the following Proposition.

\begin{prop} \label{prop2.1}
Let $-3<\gamma<0$, then for any suitable functions $f$ and $m\in\mathbb N_{+}$, we have the following estimates with weights $\omega_{t}$ for $t\ge0$
\begin{align}
 \left|\left(\omega_{t}  \nabla_{\mathcal H_{+}}^{m}\Gamma(f, f), \  \omega_{t}  \nabla_{\mathcal H_{+}}^{m}f\right)_{L^{2}(\mathbb R^{3})}\right|\le C_{4}\sum_{p=0}^{m}\binom{m}{p}\left\|\nabla^{p}_{\mathcal H_{+}}f\right\|_{2, \frac\gamma2}\left\|\omega_{t} \nabla^{m-p}_{\mathcal H_{+}}f\right\|_{A}\left\|\omega_{t} \nabla^{m}_{\mathcal H_{+}}f\right\|_{A}, \label{2.7}
\end{align}
\begin{align}
 \left|\left(\omega_{t} \nabla_{\mathcal H_{+}}^{m}\mathcal L_{2}f,\  \omega_{t} \nabla_{\mathcal H_{+}}^{m}f\right)_{L^{2}(\mathbb R^{3})}\right|\le C_{6}\sum_{p=0}^{m}\binom{m}{p}\left(\sqrt{C_{5}}\right)^{m-p}\sqrt{(m-p+3)!}\left\|\nabla^{p}_{\mathcal H_{+}}f\right\|_{2, \frac\gamma2}\left\|\omega_{t} \nabla^{m}_{\mathcal H_{+}}f\right\|_{A}, \label{2.8}
\end{align}
\begin{align}
\begin{split}
&\left|\left([\omega_{t} \nabla_{\mathcal H_{+}}^{m},\ \mathcal L_{1}]f,\ \omega_{t} \nabla_{\mathcal H_{+}}^{m}f\right)_{L^{2}(\mathbb R^{3})}\right|\le\frac18\left\|\omega_{t}\nabla^{m}_{\mathcal H_{+}}f\right\|^{2}_{A} \\
&\qquad+C_{7}\sum_{p=0}^{m-1}\binom{m}{p}\sqrt{\frac{(m-p+2)!}{2}}\left\|\omega_{t}\nabla^{p}_{\mathcal H_{+}}f\right\|_{A}\left\|\omega_{t}\nabla^{m}_{\mathcal H_{+}}f\right\|_{A}\\
&\qquad+C_{7}\sum_{p=1}^{m}\binom{m}{p}p\sqrt{\frac{(m-p+2)!}{2}}\left\|\omega_{t}\nabla^{p-1}_{\mathcal H_{+}}f\right\|_{A}\left\|\omega_{t}\nabla^{m}_{\mathcal H_{+}}f\right\|_{A}, \label{2.9}
\end{split}
\end{align}
where $C_{4}-C_{7}$ are positive constants, which depend on $\gamma$, $b$ and $c_{0}$.
\end{prop}
We will prove this Proposition in the next section.

\section{Analysis of Landau operators}\label{section commutators}

In this section, we give exact proofs of Proposition \ref{prop2.1}, which can be used directly in Section \ref{sect6}. In the following, we use the convention of implicit summation over repeated indices. We begin with proving \eqref{2.7}. Reviewing that
\begin{lemma}(\cite{LX-5})\label{representations}
For $f, g\in\mathcal S(\mathbb R^{3}_{v})$, we have
$$\mathcal L_{1}f=A_{+, j}\left((a_{jk}*\mu)A_{-, k}f\right),\ \mathcal L_{2}f=-A_{+, j}\left(\sqrt\mu(a_{jk}*(\sqrt\mu A_{-, k}f))\right),$$
$$\Gamma(f, g)=A_{+, j}\left((a_{jk}*(\sqrt\mu f))A_{+, k}g\right)-A_{+, j}\left((a_{jk}*(\sqrt\mu A_{+, k}f))g\right).$$
\end{lemma}
\begin{lemma}(\cite{LX-5})\label{Leibniz-type}
For any $m\in\mathbb N$, we have
$$\nabla^{m}_{\mathcal H_{+}}\left[\left(a_{jk}*\left(\sqrt\mu F\right)\right)G\right]=\sum_{p=0}^{m}\binom{m}{p}\left(a_{jk}*\left(\sqrt\mu \nabla^{p}_{\mathcal H_{+}}F\right)\right)\nabla^{m-p}_{\mathcal H_{+}}G.$$
\end{lemma}
In reviewing the above representations for Landau operators and the Leibniz-type formula, we need the following commutator between the nonlinear Landau operator $\Gamma$ and weights $\omega_{t}$.
\begin{lemma}\label{commutator-nonlinear}
     Let $-3<\gamma<0$, then there exists a constant $C_{4}>0$, which depends on $\gamma, \ b$ and $c_{0}$, such that for any suitable functions $f, g$ and $h$
\begin{align*}
     \left|\left(\omega_{t}\Gamma(f, g), \omega_{t} h\right)_{L^{2}(\mathbb R^{3})}\right|\le C_{4}\left\|f\right\|_{2, \frac\gamma2}\left\|\omega_{t} g\right\|_{A}\left\|\omega_{t} h\right\|_{A}.
\end{align*}
\end{lemma}
\begin{proof}
Since
$$\left(\omega_{t}\Gamma(f, g), \omega_{t} h\right)_{L^{2}(\mathbb R^{3})}=\left(\omega_{t}\Gamma(f, g), \omega_{t} h\right)_{L^{2}(\mathbb R^{3})}-\left(\Gamma(f, \omega_{t} g), \omega_{t} h\right)_{L^{2}(\mathbb R^{3})}+\left(\Gamma(f, \omega_{t} g), \omega_{t} h\right)_{L^{2}(\mathbb R^{3})},$$
for the third term, it can be estimated from Proposition \ref{trilinear} directly, so we only need to consider estimates of the first two terms. Using integration by parts, we obtain that
\begin{align*}
     &\left(\omega_{t}\Gamma(f, g), \omega_{t} h\right)_{L^{2}(\mathbb R^{3})}-\left(\Gamma(f, \omega_{t} g), \omega_{t} h\right)_{L^{2}(\mathbb R^{3})}\\
     &=\int_{\mathbb R^{3}}\left(a_{jk}*\left(\sqrt\mu f\right)\right)\partial_{k}\omega_{t}\cdot g\cdot\partial_{j}(\omega_{t} h) dv-\int_{\mathbb R^{3}}\left(a_{jk}*\left(\sqrt\mu f\right)\right)\partial_{j}\omega_{t}\cdot \partial_{k}g\cdot(\omega_{t} h) dv\\
     &\quad+\int_{\mathbb R^{3}}\left(a_{jk}*\left(\frac{v_{j}}{2}\sqrt\mu f\right)\right)\partial_{k}\omega_{t}\cdot g\cdot(\omega_{t} h) dv+\int_{\mathbb R^{3}}\left(a_{jk}*\left(\frac{v_{k}}{2}\sqrt\mu f\right)\right)\partial_{j}\omega_{t}\cdot g\cdot(\omega_{t} h) dv\\
     &\quad+\int_{\mathbb R^{3}}\left(\partial_{k}a_{jk}*\left(\sqrt\mu f\right)\right)\partial_{j}\omega_{t}\cdot g\cdot(\omega_{t} h) dv=\Gamma_{1}+\Gamma_{2}+\Gamma_{3}+\Gamma_{4}+\Gamma_{5}.
\end{align*}
Since $\partial_{j}\omega_{t}(v)=\frac{bc_{0}}{1+t}v_{j}\langle v\rangle^{b-2}\omega_{t}(v)$ for $1\le j\le 3$, then by using
\begin{align}\label{sum=0-1}
     \sum_{j}a_{jk}v_{j}=\sum_{k}a_{jk}v_{k}=\sum_{j, k}\partial_{l}a_{jk}v_{j}v_{k}=0,
\end{align}
the following formulas hold:
\begin{align*}
     \Gamma_{1}=\frac{bc_{0}}{1+t}\int_{\mathbb R^{3}}\left(a_{jk}*\left(v_{k}\sqrt\mu f\right)\right)\langle v\rangle^{b-2}\cdot(\omega_{t} g)\cdot\partial_{j}(\omega_{t} h),
\end{align*}
\begin{align*}
     \Gamma_{2}&=-\frac{bc_{0}}{1+t}\int_{\mathbb R^{3}}\left(a_{jk}*\left(v_{j}\sqrt\mu f\right)\right)\langle v\rangle^{b-2}\cdot\partial_{k}(\omega_{t} g)\cdot(\omega_{t} h)\\
     &\quad+\left(\frac{bc_{0}}{1+t}\right)^{2}\int_{\mathbb R^{3}}\left(a_{jk}*\left(v_{j}v_{k}\sqrt\mu f\right)\right)\langle v\rangle^{2(b-2)}\cdot(\omega_{t} g)\cdot(\omega_{t} h),
\end{align*}
\begin{align*}
     \Gamma_{3}=\Gamma_{4}=\frac{bc_{0}}{2(1+t)}\int_{\mathbb R^{3}}\left(a_{jk}*\left(v_{j}v_{k}\sqrt\mu f\right)\right)\langle v\rangle^{b-2}\cdot(\omega_{t} g)\cdot(\omega_{t} h),
\end{align*}
\begin{align*}
     \Gamma_{5}
     &=\frac{bc_{0}}{1+t}\int_{\mathbb R^{3}}\left(\partial_{k}a_{jk}*\left(v_{j}\sqrt\mu f\right)\right)\langle v\rangle^{b-2}\cdot(\omega_{t} g)\cdot(\omega_{t} h)\\
     &\quad-\frac{bc_{0}}{1+t}\int_{\mathbb R^{3}}\left(a_{jk}*\left(\delta_{jk}\sqrt\mu f\right)\right)\langle v\rangle^{b-2}\cdot(\omega_{t} g)\cdot(\omega_{t} h).
\end{align*}
Since $0<b\le2$, these terms can be treated as hard potential cases in the proof of the trilinear estimate to~\cite{L-4}, then we have
\begin{align*}
     \left|\Gamma_{j}\right|\lesssim\left\|f\right\|_{2, \frac\gamma2}\left\|\omega_{t} g\right\|_{A}\left\|\omega_{t} h\right\|_{A}, \quad j=1, \cdots, 5.
\end{align*}
Combining Proposition \ref{trilinear}, there exists a constant $C_{4}>0$, which depends on $\gamma,\ b$ and $c_{0}$, such that
\begin{align*}
     \left|\left(\omega_{t}\Gamma(f, g), \omega_{t} h\right)_{L^{2}(\mathbb R^{3})}\right|\le C_{4}\left\|f\right\|_{2, \frac\gamma2}\left\|\omega_{t} g\right\|_{A}\left\|\omega_{t} h\right\|_{A}.
\end{align*}
\end{proof}

\bigskip
\textbf{Proof of \eqref{2.7}.} \  From the representation of $\Gamma$ and Lemma \ref{Leibniz-type}, one has for all $m\in\mathbb N$
\begin{align} \label{Leibniz-nonlinear}
     \nabla^{m}_{\mathcal H_{+}}\Gamma(f, f)=\sum_{p=0}^{m}\binom{m}{p}\Gamma\left(\nabla^{p}_{\mathcal H_{+}}f, \nabla^{m-p}_{\mathcal H_{+}}f\right).
\end{align}
Then applying Lemma \ref{commutator-nonlinear}, we can obtain that
\begin{align*}
     \left|\left(\omega_{t}\nabla^{m}_{\mathcal H_{+}}\Gamma(f, f), \omega_{t} \nabla^{m}_{\mathcal H_{+}}f\right)_{L^{2}(\mathbb R^{3})}\right| &\le\sum_{p=0}^{m}\binom{m}{p}\left|\left(\omega_{t}\Gamma\left(\nabla^{p}_{\mathcal H_{+}}f, \nabla^{m-p}_{\mathcal H_{+}}f\right), \omega_{t} \nabla^{m}_{\mathcal H_{+}}f\right)_{L^{2}(\mathbb R^{3})}\right|\\
     &\le C_{4}\sum_{p=0}^{m}\binom{m}{p}\left\|\nabla^{p}_{\mathcal H_{+}}f\right\|_{2, \frac\gamma2}\left\|\omega_{t} \nabla^{m-p}_{\mathcal H_{+}}f\right\|_{A}\left\|\omega_{t} \nabla^{m}_{\mathcal H_{+}}f\right\|_{A}.
\end{align*}

\bigskip
\textbf{Proof of \eqref{2.8}.} \ Since $\mathcal L_{2}f=-\Gamma(f, \sqrt\mu)$, then follows immediately from \eqref{2.7} that
\begin{align*}
     &\left|\left(\omega_{t}\nabla^{m}_{\mathcal H_{+}}\mathcal L_{2}f, \omega_{t} \nabla^{m}_{\mathcal H_{+}}f\right)_{L^{2}(\mathbb R^{3})}\right|\le C_{4}\sum_{p=0}^{m}\binom{m}{p}\left\|\nabla^{p}_{\mathcal H_{+}}f\right\|_{2, \frac\gamma2}\left\|\omega_{t} \nabla^{m-p}_{\mathcal H_{+}}\sqrt\mu\right\|_{A}\left\|\omega_{t} \nabla^{m}_{\mathcal H_{+}}f\right\|_{A}.
\end{align*}
To estimate $\left\|\omega_{t} \nabla^{m-p}_{\mathcal H_{+}}\sqrt\mu\right\|^{2}_{A}$, we can write it as
\begin{align*}
     \left\|\omega_{t} \nabla^{m-p}_{\mathcal H_{+}}\sqrt\mu\right\|_{A}^{2}=\sum_{|\alpha|=m-p}\frac{(m-p)!}{\alpha!}\left\|\omega_{t} A_{+}^{\alpha}\sqrt\mu\right\|_{A}^{2}=\sum_{|\alpha|=m-p}\frac{(m-p)!}{\alpha!}\left\|\omega_{t}\mu^{-\frac12}\partial^{\alpha}_{v}\mu\right\|_{A}^{2}.
\end{align*}
Since
\begin{align*}
     \partial^{\alpha}_{v}\mu(v)=\prod_{q=1}^{3}\sum_{l=0}^{[\alpha_{q}/2]}\frac{\alpha_{q}!}{(\alpha_{q}-2l)!l!}(-1)^{\alpha_{q}-l}\left(\frac12\right)^{l}v_{q}^{\alpha_{q}-2l}\mu(v),
\end{align*}
this leads to
\begin{align*}
     \left\|\omega_{t}\mu^{-\frac12}\partial^{\alpha}_{v}\mu\right\|_{A}^{2}&\lesssim\left\|\langle\cdot\rangle^{\gamma+2}\omega_{t}\mu^{-\frac12}\partial^{\alpha+e_{j}}_{v}\mu\right\|_{L^{2}(\mathbb R^{3})}^{2}+\left\|\langle\cdot\rangle^{\gamma+4}\omega_{t}\mu^{-\frac12}\partial^{\alpha}_{v}\mu\right\|_{L^{2}(\mathbb R^{3})}^{2}\\
     &\lesssim2^{m-p+1}\left[\frac{m-p+1}{2}\right]!\left\|\langle\cdot\rangle^{\gamma+3+m-p}\omega_{t}\mu^{\frac12}\right\|_{L^{2}(\mathbb R^{3})}^{2}\\
     &\quad+2^{m-p}\left[\frac{m-p}{2}\right]!\left\|\langle\cdot\rangle^{\gamma+4+m-p}\omega_{t}\mu^{\frac12}\right\|_{L^{2}(\mathbb R^{3})}^{2}\\
     &\lesssim\left(\frac{2}{\frac14-\frac{c_{0}}{1+t}}\right)^{m-p}(m-p+2)!,
\end{align*}
here $\frac{c_{0}}{1+t}\le c_{0}\le\frac14$, and we use the fact that $p!q!\le(p+q)!\le2^{p+q}p!q!$, one has
$$\sum_{|\alpha|=m-p}\frac{(m-p)!}{\alpha!}\le2^{m-p}\frac{(m-p+1)(m-p+2)}{2},$$
then, we can obtain that
\begin{align*}
     \left\|\omega_{t} \nabla^{m-p}_{\mathcal H_{+}}\sqrt\mu\right\|_{A}^{2}\le\left(C_{5}\right)^{m-p}(m-p+3)!.
\end{align*}
Therefore, there exist positive constants $C_{5}$ and $C_{6}$, depend on $\gamma$, $b$ and $c_{0}$, such that
\begin{align*}
     &\left|\left(\omega_{t}\nabla^{m}_{\mathcal H_{+}}\mathcal L_{2}f, \omega_{t} \nabla^{m}_{\mathcal H_{+}}f\right)_{L^{2}(\mathbb R^{3})}\right|\le C_{6}\sum_{p=0}^{m}\binom{m}{p}\left(\sqrt{C_{5}}\right)^{m-p}\sqrt{(m-p+3)!}\left\|\nabla^{p}_{\mathcal H_{+}}f\right\|_{2, \frac\gamma2}\left\|\omega_{t} \nabla^{m}_{\mathcal H_{+}}f\right\|_{A}.
\end{align*}
Specially, for $m=0$, we remark that
\begin{align}\label{L2-weighted}
     &\left|\left(\omega_{t}\mathcal L_{2}f, \omega_{t} f\right)_{L^{2}(\mathbb R^{3})}\right|\le C_{6}\left\|f\right\|_{2, \frac\gamma2}\left\|\omega_{t} f\right\|_{A}.
\end{align}

\bigskip
\textbf{Proof of \eqref{2.9}.} \ For all $m\in\mathbb N_{+}$, by using \eqref{sum=0-1} and Lemma \ref{Leibniz-type}, then follows from integration by parts that
\begin{align*}
     &\left([\omega_{t} \nabla_{\mathcal H_{+}}^{m},\ \mathcal L_{1}]f,\ \omega_{t} \nabla_{\mathcal H_{+}}^{m}f\right)_{L^{2}(\mathbb R^{3})}\\
     &=\sum_{p=0}^{m-1}\binom{m}{p}\int_{\mathbb R^{3}}\left(a_{jk}*\left(\sqrt\mu\nabla^{m-p}_{\mathcal H_{+}}\sqrt\mu\right)\right)A_{-, k}\left(\omega_{t}\nabla^{p}_{\mathcal H_{+}}f\right)A_{-, j}\left(\omega_{t} \nabla^{m}_{\mathcal H_{+}}f\right)\\
     &\quad+\frac{bc_{0}}{1+t}\sum_{p=0}^{m}\binom{m}{p}\int_{\mathbb R^{3}}\left(a_{jk}*\left(v_{k}\sqrt\mu\nabla^{m-p}_{\mathcal H_{+}}\sqrt\mu\right)\right)\langle v\rangle^{b-2}\omega_{t}\nabla^{p}_{\mathcal H_{+}}f\cdot A_{-, j}\left(\omega_{t} \nabla^{m}_{\mathcal H_{+}}f\right)\\
     &\quad+\frac{bc_{0}}{1+t}\sum_{p=0}^{m}\binom{m}{p}\int_{\mathbb R^{3}}\left(a_{jk}*\left(v_{j}\sqrt\mu\nabla^{m-p}_{\mathcal H_{+}}\sqrt\mu\right)\right)\langle v\rangle^{b-2}A_{-, k}\left(\omega_{t} \nabla^{p}_{\mathcal H_{+}}f\right)\cdot\omega_{t}\nabla^{m}_{\mathcal H_{+}}f\\
     &\quad-\left(\frac{bc_{0}}{1+t}\right)^{2}\sum_{p=0}^{m}\binom{m}{p}\int_{\mathbb R^{3}}\left(a_{jk}*\left(v_{j}v_{k}\sqrt\mu\nabla^{m-p}_{\mathcal H_{+}}\sqrt\mu\right)\right)\langle v\rangle^{2(b-2)}\omega_{t} \nabla^{p}_{\mathcal H_{+}}f\cdot\omega_{t}\nabla^{m}_{\mathcal H_{+}}f\\
     &\quad-\sum_{p=1}^{m}\binom{m}{p}p\int_{\mathbb R^{3}}\left(a_{jk}*\left(\sqrt\mu\nabla^{m-p}_{\mathcal H_{+}}\sqrt\mu\right)\right)\omega_{t}\nabla^{p-1}_{\mathcal H_{+}}f\cdot\left(\omega_{t}\nabla^{m}_{\mathcal H_{+}}f\right)\\
     &\quad-\frac{bc_{0}}{1+t}\sum_{p=1}^{m}\binom{m}{p}p\int_{\mathbb R^{3}}\left(a_{jk}*\left(v_{j}\sqrt\mu\nabla^{m-p}_{\mathcal H_{+}}\sqrt\mu\right)\right)\langle v\rangle^{b-2}\omega_{t}\nabla^{p-1}_{\mathcal H_{+}}f\cdot \omega_{t}\nabla^{m}_{\mathcal H_{+}}f\\
     &=I_{1}+I_{2}+I_{3}+I_{4}+I_{5}+I_{6},
\end{align*}
here we use
$$\nabla_{\mathcal H_{+}}^{p}A_{-, k}=A_{-, k}\nabla_{\mathcal H_{+}}^{p}-p\nabla_{\mathcal H_{+}}^{p-1}, \ p\ge1 \quad {\rm and} \quad \left[\omega_{t}, A_{\pm, k}\right]=\pm\partial_{k}\omega_{t}, \ 1\le k\le3.$$
For the term $I_{1}$, decomposing $\mathbb R^{3}\times\mathbb R^{3}$  into
\begin{equation*}
\left\{|v|\le1\right\}\cup\left\{2|v'|\ge|v|, |v|\ge1\right\}\cup\left\{2|v'|\le|v|, |v|\ge1\right\}=\Omega_{1}\cup\Omega_{2}\cup\Omega_{3}.
\end{equation*}
In $\Omega_{1}\cup\Omega_{2}$, we can deduce that
\begin{align*}
     \left|a_{jk}*\left(\sqrt\mu\nabla^{m-p}_{\mathcal H_{+}}\sqrt\mu\right)\right|\lesssim\left\|\nabla^{m-p}_{\mathcal H_{+}}\sqrt\mu\right\|_{L^{2}(\mathbb R^{3})}\langle v\rangle^{\gamma},
\end{align*}
where we use \eqref{singular integration}, the fact $|v|\le1$ in $\Omega_{1}$ and $\sqrt\mu(v')\lesssim e^{-\frac{|v|^{2}}{16}}$ in $\Omega_{2}$. Then follows immediately from the Cauchy-Schwarz inequality and \eqref{2-2-1} that
\begin{align*}
     \left|I_{1}\big|_{\Omega_{1}\cup\Omega_{2}}\right|\lesssim\sum_{p=0}^{m-1}\binom{m}{p}\left\|\nabla^{m-p}_{\mathcal H_{+}}\sqrt\mu\right\|_{L^{2}(\mathbb R^{3})}\left\|\omega_{t}\nabla^{p}_{\mathcal H_{+}}f\right\|_{A}\left\|\omega_{t}\nabla^{m}_{\mathcal H_{+}}f\right\|_{A}.
\end{align*}
In $\Omega_{3}$, from \eqref{sum=0-1} and Taylor's expansion
\begin{align*}
     a_{jk}(v-v')=a_{jk}(v)+\sum_{1\le l\le3}\partial_{l}a_{jk}(v)v'_{l}+ \frac12\sum_{1\le l, p\le3}\int_{0}^{1}\partial_{lp}a_{jk}(v-sv')dsv'_{l}v'_{p},
\end{align*}
we can get that
\begin{align*}
     &\int_{\Omega_{3}}\left(a_{jk}*\left(\sqrt\mu\nabla^{m-p}_{\mathcal H_{+}}\sqrt\mu\right)\right)A_{-, k}\left(\omega_{t}\nabla^{p}_{\mathcal H_{+}}f\right)A_{-, j}\left(\omega_{t} \nabla^{m}_{\mathcal H_{+}}f\right)dv\\
     &=\int_{\Omega_{3}}a_{jk}(v)\sqrt\mu(v')\nabla^{m-p}_{\mathcal H_{+}}\sqrt\mu(v')(\mathbf I-\mathbf P_v)A_{-, k}\left(\omega_{t}\nabla^{p}_{\mathcal H_{+}}f\right)(v)(\mathbf I-\mathbf P_v)A_{-, j}\left(\omega_{t} \nabla^{m}_{\mathcal H_{+}}f\right)(v)dv\\
     &\quad+\sum_{1\le l\le3}\int_{\Omega_{3}}\partial_{l}a_{jk}(v)v'_{l}\sqrt\mu(v')\nabla^{m-p}_{\mathcal H_{+}}\sqrt\mu(v')\Big[\mathbf P_vA_{-, k}\left(\omega_{t}\nabla^{p}_{\mathcal H_{+}}f\right)(v)(\mathbf I-\mathbf P_v)A_{-, j}\left(\omega_{t} \nabla^{m}_{\mathcal H_{+}}f\right)(v)\\
     &\qquad+(\mathbf I-\mathbf P_v)A_{-, k}\left(\omega_{t}\nabla^{p}_{\mathcal H_{+}}f\right)(v)\mathbf P_vA_{-, j}\left(\omega_{t} \nabla^{m}_{\mathcal H_{+}}f\right)(v)\\
     &\qquad+(\mathbf I-\mathbf P_v)A_{-, k}\left(\omega_{t}\nabla^{p}_{\mathcal H_{+}}f\right)(v)(\mathbf I-\mathbf P_v)A_{-, j}\left(\omega_{t} \nabla^{m}_{\mathcal H_{+}}f\right)(v)\Big]dv\\
     &\quad+\frac12\sum_{1\le l, p\le3}\int_{\Omega_{3}}\int_{0}^{1}\partial_{lp}a_{jk}(v-sv')dsv'_{l}v'_{p}\sqrt\mu(v')\nabla^{m-p}_{\mathcal H_{+}}\sqrt\mu(v')A_{-, k}\left(\omega_{t}\nabla^{p}_{\mathcal H_{+}}f\right)(v)A_{-, j}\left(\omega_{t} \nabla^{m}_{\mathcal H_{+}}f\right)(v)dv,
\end{align*}
noting that $|\partial_{lp}a_{jk}(v-sv')|\lesssim\langle v\rangle^{\gamma}$ in $\Omega_{3}$, then follows immediately from the Cauchy-Schwarz inequality and \eqref{2-1-1}, \eqref{2-2-1} that
\begin{align*}
     \left|I_{1}\big|_{\Omega_{3}}\right|\lesssim\sum_{p=0}^{m-1}\binom{m}{p}\left\|\nabla^{m-p}_{\mathcal H_{+}}\sqrt\mu\right\|_{L^{2}(\mathbb R^{3})}\left\|\omega_{t}\nabla^{p}_{\mathcal H_{+}}f\right\|_{A}\left\|\omega_{t}\nabla^{m}_{\mathcal H_{+}}f\right\|_{A}.
\end{align*}
For the term $I_{2}$, we can write it as
\begin{align*}
     I_{2}&=\frac{bc_{0}}{1+t}\sum_{p=0}^{m-1}\binom{m}{p}\int_{\mathbb R^{3}}\left(a_{jk}*\left(v_{k}\sqrt\mu\nabla^{m-p}_{\mathcal H_{+}}\sqrt\mu\right)\right)\langle v\rangle^{b-2}\omega_{t}\nabla^{p}_{\mathcal H_{+}}f\cdot A_{-, j}\left(\omega_{t} \nabla^{m}_{\mathcal H_{+}}f\right)\\
     &\quad+\frac{bc_{0}}{1+t}\int_{\mathbb R^{3}}\left(a_{jk}*\left(v_{k}\mu\right)\right)\langle v\rangle^{b-2}\omega_{t}\nabla^{m}_{\mathcal H_{+}}f\cdot A_{-, j}\left(\omega_{t}\nabla^{m}_{\mathcal H_{+}}f\right)=I_{2, 1}+I_{2, 2},
\end{align*}
Since $0<b\le2$, then as discussion in $I_{1}$, we can deduce that
\begin{align*}
     &\left|I_{2, 1}\right|\lesssim\sum_{p=0}^{m-1}\binom{m}{p}\left\|\nabla^{m-p}_{\mathcal H_{+}}\sqrt\mu\right\|_{L^{2}(\mathbb R^{3})}\left\|\omega_{t}\nabla^{p}_{\mathcal H_{+}}f\right\|_{A}\left\|\omega_{t}\nabla^{m}_{\mathcal H_{+}}f\right\|_{A}.
\end{align*}
To estimate $I_{2, 2}$, using \eqref{sum=0-1}, we can write it as
\begin{align*}
     I_{2, 2}&=\frac{bc_{0}}{1+t}\int_{\mathbb R^{3}}\left(a_{jk}*\left(v_{k}\mu\right)\right)\langle v\rangle^{b-2}\omega_{t}\nabla^{m}_{\mathcal H_{+}}f\cdot (\mathbf I-\mathbf P_{v})A_{-, j}\left(\omega_{t}\nabla^{m}_{\mathcal H_{+}}f\right)
\end{align*}
since $\gamma+2>-1$, from \eqref{singular integration} we have
\begin{align*}
     \left|a_{jk}*\left(v_{k}\mu\right)\right|&\le\int_{\mathbb R^{3}}\left|v-v'\right|^{\gamma+2}\langle v\rangle\mu(v')dv'
     \le\left\|\langle \cdot\rangle\sqrt\mu\right\|_{L^{\infty}(\mathbb R^{3})}\int_{\mathbb R^{3}}\left|v-v'\right|^{\gamma+2}e^{-\frac{|v'|^{2}}{4}}dv'\le C_{0}\langle v\rangle^{\gamma+2}.
\end{align*}
then follows immediately from \eqref{2-1}, \eqref{2-1-1} and the Cauchy-Schwarz inequality that
\begin{align*}
     \left|I_{2, 2}\right|\le \frac{bc_{0}}{1+t}\frac{\tilde C_{7}}{(C_{1})^{2}}\left\|\omega_{t}\nabla^{m}_{\mathcal H_{+}}f\right\|^{2}_{A}\le\frac{1}{32}\left\|\omega_{t}\nabla^{m}_{\mathcal H_{+}}f\right\|^{2}_{A},
\end{align*}
if $\frac{bc_{0}\tilde C_{7}}{(C_{1})^{2}}\le\frac{1}{32}$. Similarly, we can deduce that
\begin{align*}
     &\left|I_{3}\right|\lesssim\sum_{p=0}^{m-1}\binom{m}{p}\left\|\nabla^{m-p}_{\mathcal H_{+}}\sqrt\mu\right\|_{L^{2}(\mathbb R^{3})}\left\|\omega_{t}\nabla^{p}_{\mathcal H_{+}}f\right\|_{A}\left\|\omega_{t}\nabla^{m}_{\mathcal H_{+}}f\right\|_{A}+\frac{1}{32}\left\|\omega_{t}\nabla^{m}_{\mathcal H_{+}}f\right\|^{2}_{A}.
\end{align*}
For term $I_{4}$, we can write it as
\begin{align*}
     I_{4}&=-\left(\frac{bc_{0}}{1+t}\right)^{2}\sum_{p=0}^{m-1}\binom{m}{p}\int_{\mathbb R^{3}}\left(a_{jk}*\left(v_{j}v_{k}\sqrt\mu\nabla^{m-p}_{\mathcal H_{+}}\sqrt\mu\right)\right)\langle v\rangle^{2(b-2)}\omega_{t} \nabla^{p}_{\mathcal H_{+}}f\cdot\omega_{t}\nabla^{m}_{\mathcal H_{+}}f\\
     &\quad-\left(\frac{bc_{0}}{1+t}\right)^{2}\int_{\mathbb R^{3}}\left(a_{jk}*\left(v_{j}v_{k}\mu\right)\right)\langle v\rangle^{2(b-2)}\omega_{t}\nabla^{m}_{\mathcal H_{+}}f\cdot\omega_{t}\nabla^{m}_{\mathcal H_{+}}f=I_{4, 1}+I_{4, 2},
\end{align*}
by using \eqref{2-2} and \eqref{singular integration}, one shows
\begin{align*}
     \left|I_{4, 1}\right|&\lesssim\sum_{p=0}^{m-1}\binom{m}{p}\left\|\nabla^{m-p}_{\mathcal H_{+}}\sqrt\mu\right\|_{L^{2}(\mathbb R^{3})}\int_{\mathbb R^{3}}\langle v\rangle^{\gamma+2}\left|\omega_{t} \nabla^{p}_{\mathcal H_{+}}f\cdot\omega_{t}\nabla^{m}_{\mathcal H_{+}}f\right|dv\\
     &\lesssim\sum_{p=0}^{m-1}\binom{m}{p}\left\|\nabla^{m-p}_{\mathcal H_{+}}\sqrt\mu\right\|_{L^{2}(\mathbb R^{3})}\left\|\omega_{t}\nabla^{p}_{\mathcal H_{+}}f\right\|_{A}\left\|\omega_{t}\nabla^{m}_{\mathcal H_{+}}f\right\|_{A},
\end{align*}
as discuss in $I_{2, 2}$, we have
\begin{align*}
     I_{4, 2}&\le\frac{1}{16}\left\|\omega_{t}\nabla^{m}_{\mathcal H_{+}}f\right\|^{2}_{A}.
\end{align*}
It remains to consider $I_{5}$ and $I_{6}$. By using \eqref{singular integration}, one has
\begin{align*}
     \left|a_{jk}*\left(\sqrt\mu\nabla^{m-p}_{\mathcal H_{+}}\sqrt\mu\right)\right|+\left|a_{jk}*\left(v_{j}\sqrt\mu\nabla^{m-p}_{\mathcal H_{+}}\sqrt\mu\right)\right|\lesssim\left\|\nabla^{m-p}_{\mathcal H_{+}}\sqrt\mu\right\|_{L^{2}(\mathbb R^{3})}\langle v\rangle^{\gamma+2},
\end{align*}
then from \eqref{2-1} and Cauchy inequality, we have
\begin{align*}
     \left|I_{6}+I_{7}\right|\lesssim\sum_{p=1}^{m}\binom{m}{p}p\left\|\nabla^{m-p}_{\mathcal H_{+}}\sqrt\mu\right\|_{L^{2}(\mathbb R^{3})}\left\|\omega_{t}\nabla^{p-1}_{\mathcal H_{+}}f\right\|_{A}\left\|\omega_{t}\nabla^{m}_{\mathcal H_{+}}f\right\|_{A}.
\end{align*}
To estimate $\|\nabla^{m-p}_{\mathcal H_{+}}\sqrt\mu\|_{L^{2}(\mathbb R^{3})}$, we have
\begin{align*}
     \left\|\nabla^{m-p}_{\mathcal H_{+}}\sqrt\mu\right\|^{2}_{L^{2}(\mathbb R^{3})}=\sum_{|\alpha|=m-p}\frac{(m-p)!}{\alpha!}\left\|A^{\alpha}_{+}\sqrt\mu\right\|^{2}_{L^{2}(\mathbb R^{3})}=\sum_{|\alpha|=m-p}(m-p)!\left\|\Psi_{\alpha}\right\|^{2}_{L^{2}(\mathbb R^{3})}=\frac{(m-p+2)!}{2},
\end{align*}
where $\{\Psi_{\alpha}\}_{\alpha\in\mathbb N^{3}}$ is the orthonormal basis in $L^{2}(\mathbb R^{3})$. Combining the above results, there exists a constant $C_{7}>0$, depends on $\gamma$, $b$ and $c_{0}$, such that
\begin{align*}
     &\left|\left([\omega_{t}  \nabla_{\mathcal H_{+}}^{m},\ \mathcal L_{1}]f,\ \omega_{t}  \nabla_{\mathcal H_{+}}^{m}f\right)_{L^{2}(\mathbb R^{3})}\right|\\
     &\le\frac18\left\|\omega_{t}\nabla^{m}_{\mathcal H_{+}}f\right\|^{2}_{A}+C_{7}\sum_{p=0}^{m-1}\binom{m}{p}\sqrt{\frac{(m-p+2)!}{2}}\left\|\omega_{t}\nabla^{p}_{\mathcal H_{+}}f\right\|_{A}\left\|\omega_{t}\nabla^{m}_{\mathcal H_{+}}f\right\|_{A}\\
     &\quad+C_{7}\sum_{p=1}^{m}\binom{m}{p}p\sqrt{\frac{(m-p+2)!}{2}}\left\|\omega_{t}\nabla^{p-1}_{\mathcal H_{+}}f\right\|_{A}\left\|\omega_{t}\nabla^{m}_{\mathcal H_{+}}f\right\|_{A}.
\end{align*}
Specially, for $m=0$, we remark that
\begin{align}\label{upper bounded}
     \left|\left([\omega_{t},\ \mathcal L_{1}]f,\ \omega_{t} f\right)_{L^{2}(\mathbb R^{3})}\right|&\le\frac18\left\|\omega_{t} f\right\|^{2}_{A}.
\end{align}

\section{Existence and Energy estimate}

Let $c_{0}$ be a positive constant, we will show the existence and $L^{2}-$estimate concerning the following weighted function
\begin{align*}
\omega_{t}(v) = e^{\frac{c_{0}}{1+t}\langle v\rangle^{b}},
\end{align*}
with $0<b\le2$ and $t\ge 0$.
\begin{prop}
\label{existence2}
    Let $-3<\gamma<0$, assume that there exists $\varepsilon_{0}>0$ such that $e^{c_{0}\langle\cdot\rangle^{b}}f_{0}\in L^{2}(\mathbb R^{3})$ satisfying
\begin{align}\label{initial-small}
     \left\|e^{c_{0}\langle\cdot\rangle^{b}}f_{0}\right\|_{L^{2}(\mathbb R^{3})}\le\varepsilon_{0},
\end{align}
then the Cauchy problem \eqref{1-2} admits a weak solution $f\in L^{\infty}(]0, \infty[; L^{2}(\mathbb R^{3}))$, and for any $T>0$
\begin{align}\label{energy-weighted1}
        &\left\|\omega_{t} f(t)\right\|^{2}_{L^{2}(\mathbb R^{3})}+\frac{2c_{0}}{(1+T)^{2}}\int_{0}^{t}\left\|\langle\cdot\rangle^{\frac b2}\omega_{\tau} f\right\|^{2}_{L^{2}(\mathbb R^{3})}d\tau+\int_{0}^{t}\|\omega_{\tau} f\|^{2}_{A}d\tau\le C_{T} \varepsilon_0^{2}, \quad\forall \ 0<t\le T.
\end{align}
\end{prop}
\begin{proof}
     Let $f^{0}=\omega_{0}f_{0}$, consider the following approximate equations
\begin{align*}
     \partial_tf^{n}+\mathcal Lf^{n}=\Gamma(f^{n-1}, f^{n}),\quad f^{n}|_{t=0}=f_0, \quad n\ge1.
\end{align*}
For $n=1$, according to the general existence proof we have that the above Cauchy problem has a unique weak solution $f\in L^{\infty}([0, \infty[; L^{2}(\mathbb R^{3}))$. With the equation
$$\left(\partial_{t}+\frac{c_{0}}{(1+t)^{2}}\langle v\rangle^{b}\right)\omega_{t} f^{1}+\omega_{t}\mathcal Lf^{1}=\omega_{t} \Gamma (f^{0}, f^{1}),$$
taking the inner product with $\omega_{t} f^{1}$ on $L^{2}(\mathbb R^{2})$, then it follows from Lemma \ref{lemma2.1}, Lemma \ref{commutator-nonlinear} and \eqref{L2-weighted}, \eqref{upper bounded} and
\eqref{initial-small}, we have, for small $\varepsilon_{0}$,
\begin{align*}
     &\frac{d}{dt}\left\|\omega_{t} f^{1}\right\|^{2}_{L^{2}(\mathbb R^{3})}+\frac{2c_{0}}{(1+t)^{2}}\left\|\langle\cdot\rangle^{\frac b2}\omega_{t} f^{1}\right\|^{2}_{L^{2}(\mathbb R^{3})}+\|\omega_{t} f^{1}\|^{2}_{A}\le C\left\|\omega_{t} f^{1}\right\|^{2}_{2, \frac\gamma2},
\end{align*}
then for all $0<t\le T$
\begin{align}\label{existence1}
     &\left\|\omega_{t} f^{1}(t)\right\|^{2}_{L^{2}(\mathbb R^{3})}+\frac{2c_{0}}{(1+T)^{2}}\int_{0}^{t}\left\|\langle\cdot\rangle^{\frac b2}\omega_{\tau} f^{1}\right\|^{2}_{L^{2}(\mathbb R^{3})}d\tau+\int_{0}^{t}\|\omega_{\tau} f\|^{2}_{A}d\tau
     \le C(\varepsilon_{0})^{2}.
\end{align}
By iterating and \eqref{existence1}, one can obtain that for all $n\ge1$, the approximate equations admit a solution $f^{n}$ satisfying
\begin{align*}
        &\left\|\omega_{t} f^{n}(t)\right\|^{2}_{L^{2}(\mathbb R^{3})}+\frac{2c_{0}}{(1+T)^{2}}\int_{0}^{t}\left\|\langle\cdot\rangle^{\frac b2}\omega_{\tau} f^{n}\right\|^{2}_{L^{2}(\mathbb R^{3})}d\tau+\int_{0}^{t}\|\omega_{\tau} f^{n}\|^{2}_{A}d\tau\le C(\varepsilon_{0})^{2}.
\end{align*}
Setting $\zeta^{n}=f^{n+1}-f^{n}$ for $n\ge1$, then
\begin{align*}
     \partial_t\zeta^{n}+\mathcal L\zeta^{n}=\Gamma(f^{n}, \zeta^{n})+\Gamma(\zeta^{n-1}, f^{n}),\quad \zeta^{n}|_{t=0}=0.
\end{align*}
Taking inner product with $\omega_{t}\zeta^{n}$ on both side of
$$\left(\partial_{t}+\frac{c_{0}}{(1+t)^{2}}\langle v\rangle^{b}\right)\omega_{t}\zeta^{n}+\omega_{t}\mathcal L\zeta^{n}=\omega_{t}\Gamma(f^{n}, \zeta^{n})+\omega_{t}\Gamma(\zeta^{n-1}, f^{n}),$$
for any $T>0$, by using Lemma \ref{lemma2.1}, Lemma \ref{commutator-nonlinear}, \eqref{L2-weighted} and \eqref{upper bounded}, we can obtain that, for all $0<t\le T$
\begin{align*}
     &\left\|\omega_{t}\zeta^{n}(t)\right\|^{2}_{L^{2}(\mathbb R^{3})}+\frac{2c_{0}}{(1+T)^{2}}\int_{0}^{t}\left\|\langle\cdot\rangle^{\frac b2}\omega_{\tau}\zeta^{n}\right\|^{2}_{L^{2}(\mathbb R^{3})}d\tau+\int_{0}^{t}\left\|\omega_{\tau}\zeta^{n}\right\|^{2}_{A}d\tau\le C\, \varepsilon_{0}^{2}  \left\|\omega_{t}\zeta^{n-1}\right\|^{2}_{L^{\infty}([0, T]; L^{2}(\mathbb R^{3}))},
\end{align*}
this yields, for small $ \varepsilon_{0}$,
\begin{align*}
     &\left\|\omega_{t}\zeta^{n}\right\|^{2}_{L^{\infty}([0, \infty[; L^{2}(\mathbb R^{3}))}\le\frac12\left\|\omega_{t}\zeta^{n-1}\right\|^{2}_{L^{\infty}([0, \infty[; L^{2}(\mathbb R^{3}))}\le2^{-n}\left\|\omega_{t}\zeta^{0}\right\|^{2}_{L^{\infty}([0, \infty[; L^{2}(\mathbb R^{3}))}
\end{align*}
Thus, $\{f^{n}\}\subset L^{2}(\mathbb R^{3})$ is a Cauchy sequence, the limit function $f\in L^{\infty}([0, \infty[; L^{2}(\mathbb R^{3}))$ is a unique solution to Cauchy problem \eqref{1-2}, and satisfying \eqref{energy-weighted1}.

\end{proof}
With the same argument, we can also obtain the following global existence of weak solutions.
\begin{cor}
\label{existence3}
    Let $-3<\gamma<0$, and $f_{0}\in L^{2}(\mathbb R^{3})$ satisfying
\begin{align}\label{initial-small1}
     \left\|f_{0}\right\|_{L^{2}(\mathbb R^{3})}\le\varepsilon_{0},
\end{align}
then the Cauchy problem \eqref{1-2} admits a global in time weak solution satisfying, for any $T>0$, there exists $C_T>0$ such that
\begin{align}\label{energy-weighted11}
        &\left\|f(t)\right\|^{2}_{L^{2}(\mathbb R^{3})}+\int_{0}^{t}\| f\|^{2}_{A}ds\le C_T\, \varepsilon^{2}_{1}, \quad\forall \ 0<t\le T.
\end{align}
\end{cor}

\section{Analyticity in time variable}
This section is devoted to proving time analyticity, for which the weight $\omega$ is not required. With the Cauchy problem \eqref{1-2}, we can establish the following $L^{2}-$esitmate.
\begin{lemma}\label{lemma3.2}
    For $-3<\gamma<0$. Lef $f$ be the solution of Cauchy problem \eqref{1-2} with $\|f\|_{L^{\infty}([0, \infty[; L^{2})}$ small enough. Then there exists a constant $C_{8}>0$ such that for any $T>0$ and $t\in]0,T]$,
    \begin{align}\label{k=1}
        \|t\partial_tf\|^2_{ L^2(\mathbb R^3)}+\int_0^t\|s\partial_sf(s)\|^2_Ads\le\left(C_8\epsilon\right)^2.
    \end{align}
\end{lemma}
\begin{proof}
    With the Cauchy problem \eqref{1-2}, we have that
for all $0<t\le T$,
    \begin{align*}
            &\frac12\|t\partial_tf\|^2_{L^2(\mathbb R^3)}+\int_0^t(\mathcal L_1(s\partial_sf), s\partial_sf)_{L^2(\mathbb R^3)}ds\\
            &=\int_0^ts\|\partial_sf\|^2_{L^2(\mathbb R^3)}ds-\int_0^t(\mathcal L_2(s\partial_sf), s\partial_sf)_{L^2(\mathbb R^3)}+\int_0^t(s\partial_s\Gamma(f, f),
        s\partial_s f)_{L^2(\mathbb R^3)}ds\\
            &=D_1+D_2+D_3.
    \end{align*}
It follows from \eqref{1-2} that
    \begin{align*}
            D_{1}&=\int_0^t(\Gamma(f, f),
        s\partial_s f)_{L^2(\mathbb R^3)}ds-\int_0^t(\mathcal L_1f, s\partial_sf)_{L^2(\mathbb R^3)}ds-\int_0^t(\mathcal L_2f, s\partial_sf)_{L^2(\mathbb R^3)}ds,
    \end{align*}
since $\mathcal L_{1}f=-\Gamma(\sqrt\mu, f)$ and $\mathcal L_{2}f=-\Gamma(f, \sqrt\mu)$, then follows immediately from Lemma \ref{trilinear} and \eqref{2-2} that for all $0<t\le T$
    \begin{align*}
            D_{1}
            \le C_{2}\varepsilon_{0}\int_0^t\|f(s)\|_{A}\|s\partial_sf\|_{A}ds+\tilde C_2\int_0^t\|f(s)\|_A\|s\partial_sf\|_Ads+\frac{\tilde C_2}{\sqrt{C_{1}}}\int_0^t\|f(s)\|_A\|s\partial_sf\|_Ads,
    \end{align*}
then from the Cauchy-Schwarz inequality, for any $0<\delta<1$, we have
    \begin{equation}\label{D1}
        \begin{split}
            D_{1}=\int_0^ts\|\partial_sf\|^2_{L^2(\mathbb R^3)}ds&\le\delta\int_0^t\|s\partial_sf\|^2_Ads+C_\delta\int_0^t\|f(s)\|^2_{A}ds,
        \end{split}
    \end{equation}
    with $C_\delta$ depends on $C_{1}, C_{2}$.

    For $D_2$, since $\mathcal L_{2}f=-\Gamma(f, \sqrt\mu)$, then by using Lemma \ref{trilinear} and \eqref{2-2},  we have for all $0<t\le T$
    \begin{align*}
            |D_2|\le\epsilon_2\int_0^t\|s\partial_sf\|^2_Ads+C_{\epsilon_2}T\int_0^ts\|\partial_sf\|^2_{L^2(\mathbb R^3)}ds,
    \end{align*}
    by using \eqref{D1} with $C_{\epsilon_2}T\delta\le\epsilon_2$,
    \begin{equation*}
        |D_2|\le2\epsilon_2\int_0^t\|s\partial_sf\|^2_Ads+\tilde C_{\epsilon_2}\int_0^t\|f(s)\|^2_{A}ds,
    \end{equation*}
    with $\tilde C_{\epsilon_2}$ depends on $C_{1}, C_2$ and $T$. Now, we trun to the term $S_3$, since $\gamma<0$, by using Lemma \ref{trilinear}, \eqref{energy-weighted11} and Cauchy-Schwarz inequality, one has
    \begin{align*}
            |D_3|
            &\le\int^t_0s^2\left|(\Gamma(\partial_s f,f),\partial_s f)\right|ds
              +\int^t_0s^2\left|(\Gamma(f,\partial_s f),\partial_s f)\right|ds\\
            &\le C_{2}\int^t_0\|s\partial_s f\|_{L^{2}}\|f(s)\|_{A}\|s\partial_s f\|_{A}+\|f(s)\|_{L^{2}}\|s\partial_s f\|^{2}_{A}ds\\
            &\le C_{2}\varepsilon_{1}\|s\partial_s f\|^{2}_{L^{\infty}(]0, t]; L^{2})}+2C_{2}\varepsilon_{1}\int^t_0\|s\partial_s f\|^{2}_{A}ds.
    \end{align*}
    It remains to consider the operator $\mathcal L_{1}$. Since $\gamma<0$, by lemma \ref{lemma2.1}, for all $0<t\le T$, we can conclude
    \begin{equation*}
        \begin{split}
            \int_0^t(\mathcal L_1(s\partial_sf), s\partial_sf)_{L^2(\mathbb R^3)}ds
            &\ge(1-\epsilon_1)\int_0^t\|s\partial_sf\|^2_Ads-C_{\epsilon_1}\int_0^t\|s\partial_sf\|^2_{2, \frac{\gamma}{2}}ds\\
            &\ge(1-\epsilon_1)\int_0^t\|s\partial_sf\|^2_Ads-TC_{\epsilon_1}\int_0^ts\|\partial_sf\|^2_{L^2(\mathbb R^3)}ds,
        \end{split}
    \end{equation*}
    applying \eqref{D1} with $TC_{\epsilon_1}\delta\le\epsilon_{1}$, one has
    \begin{align*}
            &\int_0^t(\mathcal L_1(s\partial_sf), s\partial_sf)_{L^2(\mathbb R^3)}ds\ge(1-2\epsilon_1)\int_0^t\|s\partial_sf\|^2_Ads-\tilde C_{\epsilon_1}\int_0^t\|f(s)\|^2_{A}ds.
    \end{align*}
    Combining these results and taking
    $$2\epsilon_{1}=\delta=2\epsilon_{2}=2C_{2}\varepsilon_{1}=\frac18,$$
    then we can get that for all $0<t\le T$
    \begin{align*}
            &\|t\partial_tf\|^2_{L^2}+\int_0^t\|s\partial_sf\|^2_Ads\le 2\tilde C_{8}\int_0^t\|f(s)\|^2_{A}ds+2(C_{2}\varepsilon_{1})^{2}\|s\partial_s f\|^{2}_{L^{\infty}(]0, t]; L^{2})},
    \end{align*}
    here $\tilde C_{8}$ depends on $C_{1}, C_{2}$ and $T$. Now, we choose  $2(C_{2}\varepsilon_{1})^{2}\le\frac12$, then
    $$\frac12\|t\partial_tf\|^2_{L^{\infty}(]0, T]; L^{2})}\le2\tilde C_{8}(\varepsilon_{1})^{2},$$
    from this, one can deduce that for all $0<t\le T$
    \begin{equation*}
        \begin{split}
            &\|t\partial_tf\|^2_{L^2(\mathbb R^3)}+\int_0^t\|s\partial_sf\|^2_Ads\le 4\tilde C_{8}(\varepsilon_{1})^{2}\le (C_{8}\epsilon)^{2},
        \end{split}
    \end{equation*}
    here $C_{8}\ge2\sqrt{\tilde C_{8}}$ depends on $C_{1}, C_2$ and $T$.
\end{proof}

Now, we turn to show the time analyticity.
\begin{prop}\label{prop 4.1}
    For $-3<\gamma<0$. Let $f$ be the solution of Cauchy problem \eqref{1-2} with $\|f\|_{L^{\infty}([0, \infty[; L^{2})}$ small enough. Then there exists a constant $B>0$ such that for any $T>0$, $t\in]0,T]$ and $k\in\mathbb N$,
    \begin{equation}\label{4-1}
        \|t^k\partial_t^kf(t)\|^2_{L^2(\mathbb R^3)}+\int_0^t\|s^k\partial_s^kf(s)\|^2_Ads\le\left(B^{k-2}(k-1)!\right)^2.
    \end{equation}
\end{prop}
\begin{proof}
    We prove this proposition by induction on the index $k$. For $k=0, 1$, it is enough to take in \eqref{energy-weighted11} and \eqref{k=1}. Assume \eqref{4-1} holds true, for any $1\le m\le k-1$ with $k\ge2$,
    \begin{equation}\label{induction-m}
        \|t^m\partial_t^mf(t)\|^2_{L^2(\mathbb R^3)}+\int_0^t\|s^m\partial_s^mf(s)\|^2_Ads\le\left(B^{m-1}m!\right)^2.
    \end{equation}
    We shall prove \eqref{induction-m} holds for $m=k$. Since $\mu$ is the function with respect to $v$, which implies
    \begin{align*}
        t^k\partial_t^k\mathcal L_1f=\mathcal L_1(t^k\partial_t^kf), \quad t^k\partial_t^k\mathcal L_2f=\mathcal L_2(t^k\partial_t^kf).
    \end{align*}
    Then by \eqref{1-2}, we have
\begin{equation*}
\begin{split}
        \partial_t (t^k\partial^k_t f)&+\mathcal{L}_1(t^k\partial^k_t f)
        =k t^{k-1}\partial^k_t f-\mathcal{L}_2(t^k\partial^k_t f)+\Gamma(f, t^k\partial^k_t f)\\
        &+\Gamma(t^k\partial^k_t f, f)+\sum_{1\le j\le k-1}C^j_k\ \Gamma(t^{j}\partial^{j}_tf, t^{k-j}\partial^{k-j}_t f),
\end{split}
\end{equation*}
where $C^j_k=\frac{k!}{j!(k-j)!}$. Then taking the $L^2(\mathbb{R}^3)$ inner product of both sides with respect to $t^k\partial^k_tf$, we have
\begin{align*}
&\frac{1}{2}\frac{d}{dt}\|t^k\partial^k_tf\|_{L^2}^2+(\mathcal{L}_1(t^{k}\partial^k_tf), t^{k}\partial^k_tf)_{L^2}=kt^{2k-1}\|\partial^k_tf\|_{L^2}^2-(\mathcal{L}_2(t^{k}\partial^k_tf), t^{k}\partial^k_tf)_{L^2}\\
&\qquad \qquad+(\Gamma(f, t^k\partial^k_tf), t^k\partial^k_tf)_{L^2} +(\Gamma(t^k\partial^k_tf, f), t^k\partial^k_tf)_{L^2}
\\
&\qquad \qquad +\sum_{1\le j\le k-1}C^j_k\ \Gamma(t^{j}\partial^{j}_tf, t^{k-j}\partial^{k-j}_t f), t^k\partial^k_tf)_{L^2}.
\end{align*}
    For all $0<t\le T$, integrating from 0 to $t$, since $\gamma<0$, by using lemma \ref{lemma2.1}, it follows that
    \begin{align*}
            &\int_0^t(\mathcal L_1(s^k\partial_s^kf) ,s^k\partial_s^kf)_{L^2(\mathbb R^3)}ds
            \ge(1-\epsilon_1)\int_0^t\|s^k\partial_s^kf\|^2_Ads-TC_{\epsilon_1}\int_0^ts^{2k-1}\|\partial_s^kf\|^2_{L^2(\mathbb R^3)}ds,
    \end{align*}
    and since $\mathcal L_{2}=-\Gamma(f, \sqrt\mu)$, from Lemma \ref{trilinear}, we have
    \begin{equation*}
        \begin{split}
            \int_0^t|(\mathcal L_2(s^k\partial_s^kf) ,&s^k\partial_s^kf)_{L^2}|ds\le\epsilon_2\int_0^t\|s^k\partial_s^kf\|^2_Ads+TC_{\epsilon_2}\int_0^ts^{2k-1}\|\partial_s^kf\|^2_{L^2}ds.
        \end{split}
    \end{equation*}
    Combining the results above and taking
    $$\epsilon_1=\epsilon_2=\frac18,$$
    we have for all $0<t\le T$,
    \begin{align*}
            &\|t^k\partial_t^kf(t)\|^2_{L^2(\mathbb R^3)}+\frac32\int_0^t\|s^k\partial_s^kf\|^2_Ads\\
            &\le C_{9}\int_0^ts^{2k-1}\|\partial_s^kf\|^2_{L^2(\mathbb R^3)}ds+2\sum_{1\le j\le k-1}C^j_k\int_0^t\Gamma(s^{j}\partial^{j}_sf, s^{k-j}\partial^{k-j}_s f), s^k\partial^k_sf)_{L^2}ds\\
            &\quad+2\int_0^t(\Gamma(f, s^k\partial^k_sf), s^k\partial^k_sf)_{L^2}ds+2\int_0^t(\Gamma(s^k\partial^k_sf, f), s^k\partial^k_sf)_{L^2}ds
\\
&=J_{1}+J_{2}+J_{3}+J_{4},
    \end{align*}
    with $C_{9}$ depends on $C_{1}, C_{2}$ and $T$.

    For term $J_{2}$, by using Lemma \ref{trilinear} and the fact $\gamma<0$, one has
    \begin{align*}
            J_{2}&\le 2C_{2}\sum_{1\le j\le k-1}C^j_k\int_{0}^{t}\|s^{j}\partial^{j}_{s}f\|_{L^{2}}\|s^{k-j}\partial^{k-j}_{s}f\|_A\|s^k\partial^k_sf\|_{A}ds,
    \end{align*}
   then by using Cauchy-Schwarz inequality and \eqref{induction-m}, it follows that
    \begin{align*}
            J_{2}&\le2C_{2}B^{k-2}\sum_{1\le j\le k-1}C^j_k(j-2)!(k-j-2)!\left(\int_{0}^{t}\|s^k\partial^k_sf\|^{2}_{A}ds\right)^{\frac12}\\
            &\le\frac18\int_{0}^{t}\|s^k\partial^k_sf\|^{2}_{A}ds+2\left(2C_{2}B^{k-2}\sum_{1\le j\le k-1}C^j_k(j-2)!(k-j-2)!\right)^{2}\\
            &\le\frac18\int_{0}^{t}\|s^k\partial^k_sf\|^{2}_{A}ds+2\left(64C_{2}B^{k-2}(k-2)!\right)^{2},
    \end{align*}
    here we use
    $$\sum_{1\le j\le k-1}C^j_k(j-2)!(k-j-2)!\le32(k-2)!.$$
    For the term $J_{3}$ and $J_{4}$, by using \eqref{trilinear} and the fact $\gamma<0$, one has
    \begin{align*}
            J_{3}+J_{4}&\le 2C_{2}\int_{0}^{t}\left(\|f(s)\|_{L^{2}}\|s^k\partial^k_sf\|^{2}_{A}+\|f(s)\|_{A}\|s^k\partial^k_sf\|_{L^{2}}\|s^k\partial^k_sf\|_{A}\right)ds\\
            &\le2C_{2}\varepsilon_{1}\int_{0}^{t}\|s^k\partial^k_sf\|^{2}_{A}ds+2C_{2}\varepsilon_{1}\|s^{k}\partial^{k}_{s}f\|_{L^{\infty}(]0, t]; L^{2})}\left(\int_{0}^{t}\|s^k\partial^k_sf\|^{2}_{A}ds\right)^{\frac12}\\
            &\le4C_{2}\varepsilon_{1}\int_{0}^{t}\|s^k\partial^k_sf\|^{2}_{A}ds+\frac12C_{2}\varepsilon_{1}\|s^{k}\partial^{k}_{s}f\|^{2}_{L^{\infty}(]0, t]; L^{2})}.
    \end{align*}
    Then taking $4C_{2}\varepsilon_{1}\le\frac{1}{8}$, we have
    \begin{align*}
            J_{3}+J_{4}&\le\frac18\int_{0}^{t}\|s^k\partial^k_sf\|^{2}_{A}ds+\frac{1}{16}\|s^{k}\partial^{k}_{s}f\|^{2}_{L^{\infty}(]0, t]; L^{2})}.
    \end{align*}
    It remains to consider the term $J_{1}$, since $f$ is the solution of \eqref{1-2} and $k\ge2$, then
    $$
    \partial_t^kf=\partial_t^{k-1}\Gamma(f, f)-\mathcal L_1(\partial_t^{k-1}f)-\mathcal L_2(\partial_t^{k-1}f),$$
    which implies
    \begin{align*}
            J_{1}&=C_{9}\int_0^t(\Gamma(f, s^{k-1}\partial_s^{k-1}f), s^k\partial_s^kf)_{L^2}+(\Gamma(s^{k-1}\partial_s^{k-1}f, f), s^k\partial_s^kf)_{L^2}ds\\
            &\quad+C_{9}\sum_{1\le j\le k-2}C_{k-1}^{j}\int_0^t(\Gamma(s^{j}\partial_s^{j}f, s^{k-1-j}\partial_s^{k-1-j}f), s^k\partial_s^kf)_{L^2}ds\\
            &\quad-C_{9}\int_0^t(\mathcal L_1(s^{k-1}\partial_s^{k-1}f), s^k\partial_s^kf)_{L^2(\mathbb R^3)}ds-\int_0^t(\mathcal L_2(s^{k-1}\partial_s^{k-1}f), s^k\partial_s^kf)_{L^2(\mathbb R^3)}ds\\
            &=J_{1, 1}+J_{1, 2}+J_{1, 3}+J_{1, 4}.
    \end{align*}
    For $J_{1, 1}$ and $J_{1, 2}$, by using Lemma \ref{trilinear}, \eqref{induction-m}, and the fact $\gamma<0$, we have
    \begin{align*}
            &\left|J_{1, 1}+J_{1, 2}\right|
            \le2C_{2}C_{9}\varepsilon_{0} B^{k-2}(k-3)!\left(\int_0^t\|s^{k}\partial_s^{k}f\|^{2}_{A}ds\right)^{\frac12}\\
            &\quad+C_{2}C_{9}\sum_{1\le j\le k-2}C_{k-1}^{j}B^{k-3}(j-2)!(k-j-3)!\left(\int_0^t\|s^{k}\partial_s^{k}f\|^{2}_{A}ds\right)^{\frac12}\\
            &\le2\left(64C_{2}C_{9}B^{k-3}(k-2)!\right)^{2}+\left(\left(\varepsilon_{1}\right)^{2}+\frac{1}{16}\right)\int_0^t\|s^{k}\partial_s^{k}f\|^{2}_{A}ds
    \end{align*}
    From Lemma \ref{trilinear}, then applying \eqref{2-1} and \eqref{induction-m}, it follows that
     \begin{align*}
            \left|J_{1, 3}+J_{1, 4}\right|
            &\le C_{9}\left(\tilde C_{2}+\frac{\tilde C_{2}}{\sqrt{C_{1}}}\right)B^{k-2}(k-3)!\left(\int_0^t\|s^{k}\partial_s^{k}f\|^{2}_{A}ds\right)^{\frac12}\\
            &\le\frac{1}{8}\int_0^t\|s^{k}\partial_s^{k}f\|^{2}_{A}ds+2\left(C_{9}\left(\tilde C_{2}+\frac{\tilde C_{2}}{\sqrt{C_{1}}}\right)B^{k-2}(k-3)!\right)^{2}.
    \end{align*}
    So that, taking $\left(\varepsilon_{1}\right)^{2}\le\frac{1}{16}$, we can get that
    \begin{equation*}
        \begin{split}
            &J_{1}\le\frac14\int_0^t\|s^{k}\partial_s^{k}f\|^{2}_{A}ds+2\left(C_{9}\left(\tilde C_{2}+\frac{\tilde C_{2}}{\sqrt{C_{1}}}\right)B^{k-2}(k-3)!\right)^{2}+2\left(64C_{2}C_{9}B^{k-3}(k-2)!\right)^{2}
        \end{split}
    \end{equation*}
    Combining the results of $J_{1}-J_{4}$, we have for all $0<t\le T$,
    \begin{equation*}
        \begin{split}
            &\|t^k\partial_t^kf(t)\|^2_{L^2(\mathbb R^3)}+\int_0^t\|s^k\partial_s^kf\|^2_Ads\le\tilde C_{9}\left(B^{k-2}(k-2)!\right)^{2}+\frac{1}{2}\|t^{k}\partial^{k}_{t}f\|^{2}_{L^{\infty}(]0, T]; L^{2})},
        \end{split}
    \end{equation*}
    from this, we can get that for all $T>0$
    $$\frac{1}{2}\|t^{k}\partial^{k}_{t}f\|^{2}_{L^{\infty}(]0, T]; L^{2})}\le\tilde C_{9}\left(B^{k-2}(k-2)!\right)^{2},$$
    and therefore, we have for all $0<t\le T$,
    \begin{equation*}
        \begin{split}
            &\|t^k\partial_t^kf(t)\|^2_{L^2(\mathbb R^3)}+\int_0^t\|s^k\partial_s^kf\|^2_Ads\le 2\tilde C_{9}\left(B^{k-2}(k-2)!\right)^{2}.
        \end{split}
    \end{equation*}
    Finally, taking $B\ge\sqrt{2\tilde C_{9}}$, we can get that for all $0<t\le T$,
    $$\|t^k\partial_t^kf(t)\|^2_{L^2(\mathbb R^3)}+\int_0^t\|s^k\partial_s^kf(s)\|^2_Ads\le\left(B^{k-2}(k-1)!\right)^2.$$
\end{proof}

\section{Gelfand-Shilov Regularizing Effect}  \label{sect6}
In this section, we will show the Gelfand-Shilov regularizing effect with the sub-exponential weighted functions.
\begin{prop}\label{prop-v}
    Assume that $-3<\gamma<0$. Let $f$ be the solution of Cauchy problem \eqref{1-2} satisfying $\|\omega_{t} f\|_{L^{\infty}([0, \infty[; L^{2}(\mathbb R^{3}))} $ small enough. Then there exists a positive constant $L>1$, depends on $\gamma, \ b, \ c_{0}$ and $T$, such that
    \begin{align}\label{high-v}
    \begin{split}
        &\left\|\omega_{t} t^{m\sigma}\nabla_{\mathcal H_{+}}^{m} f(t)\right\|^2_{L^2(\mathbb R^3)}+\frac{c_{0}}{(1+T)^{2}}\int_{0}^{t}\left\|\omega_{\tau} \tau^{m\sigma}\langle\cdot\rangle^{\frac b2}\nabla_{\mathcal H_{+}}^{m} f(\tau)\right\|^2_{L^2(\mathbb R^3)}d\tau\\
        &\quad+\int_0^t\left\|\omega_{\tau}\tau^{m\sigma}\nabla_{\mathcal H_{+}}^{m}f(\tau)\right\|^2_{A}d\tau\le \left(L^{m-1}(m-2)!\right)^{2\sigma}, \  \forall \ m\in\mathbb N, \ \forall \ t\in]0, T].
    \end{split}
    \end{align}
\end{prop}
\begin{proof}
We prove this proposition by induction on the index $m$. For $m=0$, it has been proved in \eqref{energy-weighted1}. From \eqref{1-2} and Lemma \ref{lemma2.1}, we have
\begin{align*}
     &\frac12\frac{d}{dt}\left\|\omega_{t} t^{m\sigma}\nabla_{\mathcal H_{+}}^{m} f(t)\right\|^2_{L^2(\mathbb R^3)}+\frac{c_{0}}{(1+t)^{2}}\left\|\langle\cdot\rangle^{\frac b2}\omega_{t} t^{m\sigma}\nabla_{\mathcal H_{+}}^{m} f(t)\right\|^{2}_{L^{2}(\mathbb R^{3})}+\frac78\left\|\omega_{t} t^{m\sigma}\nabla_{\mathcal H_{+}}^{m} f(t)\right\|^{2}_{A}\\
     &\le m\sigma t^{2m\sigma-1}\left\|\omega_{t} \nabla_{\mathcal H_{+}}^{m} f(t)\right\|^{2}_{L^{2}(\mathbb R^{3})}+\left|\left(\omega_{t} t^{m\sigma}\nabla_{\mathcal H_{+}}^{m} \Gamma(f, f), \omega_{t} t^{m\sigma}\nabla_{\mathcal H_{+}}^{m} f\right)_{L^{2}(\mathbb R^{3})}\right|\\
     &\quad+\left|\left(\omega_{t}t^{m\sigma}\nabla_{\mathcal H_{+}}^{m}\mathcal L_{2}f, \omega_{t}t^{m\sigma}\nabla_{\mathcal H_{+}}^{m}f\right)_{L^{2}(\mathbb R^{3})}\right|+\left|\left(t^{m\sigma}[\omega_{t}\nabla_{\mathcal H_{+}}^{m},\ \mathcal L_{1}]f, \omega_{t}t^{m\sigma}\nabla_{\mathcal H_{+}}^{m}f\right)_{L^{2}(\mathbb R^{3})}\right|.
\end{align*}
For all $t\in]0, T]$, integrating from $0$ to $t$,
\begin{align*}
     &\left\|\omega_{t} t^{m\sigma}\nabla_{\mathcal H_{+}}^{m} f(t)\right\|^2_{L^2(\mathbb R^3)}+\frac{2c_{0}}{(1+T)^{2}}\int_{0}^{t}\left\|\langle\cdot\rangle^{\frac b2}\omega_{\tau}\tau^{m\sigma}\nabla_{\mathcal H_{+}}^{m} f(\tau)\right\|^{2}_{L^{2}(\mathbb R^{3})}d\tau+\frac32\int_{0}^{t}\left\|\omega_{\tau}\tau^{m\sigma}\nabla_{\mathcal H_{+}}^{m} f(\tau)\right\|^{2}_{A}d\tau\\
     &\le2m\sigma\int_{0}^{t}\tau^{2m\sigma-1}\left\|\omega_{\tau} \nabla_{\mathcal H_{+}}^{m} f(\tau)\right\|^{2}_{L^{2}(\mathbb R^{3})}d\tau+2C_{3}\int_{0}^{t}\left\|\omega_{\tau}\tau^{m\sigma}\nabla_{\mathcal H_{+}}^{m} f(\tau)\right\|^{2}_{2, \frac\gamma2}d\tau+Q_{1}+Q_{2}+Q_{3},
\end{align*}
with
\begin{align*}
     &Q_{1}=2C_{7}\sum_{p=0}^{m-1}\binom{m}{p}\sqrt{\frac{(m-p+2)!}{2}}\int_{0}^{t}\tau^{(m-p)\sigma}\left\|\omega_{\tau} \tau^{p\sigma}\nabla^{p}_{\mathcal H_{+}}f(\tau)\right\|_{A}\left\|\omega_{\tau} \tau^{m\sigma}\nabla^{m}_{\mathcal H_{+}}f(\tau)\right\|_{A}d\tau\\
     &\quad+2C_{7}\sum_{p=1}^{m}\binom{m}{p}p\sqrt{\frac{(m-p+2)!}{2}}\int_{0}^{t}\tau^{(m-p+1)\sigma}\left\|\omega_{\tau} \tau^{(p-1)\sigma}\nabla^{p-1}_{\mathcal H_{+}}f(\tau)\right\|_{A}\left\|\omega_{\tau} \tau^{m\sigma}\nabla^{m}_{\mathcal H_{+}}f(\tau)\right\|_{A}d\tau,
\end{align*}
and
\begin{equation*}
\begin{split}
     Q_{2}=&2\int_{0}^{t}\left|\left(\omega_{\tau} \tau^{m\sigma}\nabla^{m}_{\mathcal H_{+}}\mathcal L_{2}f(\tau), \omega_{\tau} \tau^{m\sigma}\nabla^{m}_{\mathcal H_{+}}f(\tau)\right)_{L^{2}(\mathbb R^{3})}\right|d\tau,
\end{split}
\end{equation*}
and
\begin{equation*}
\begin{split}
     Q_{3}=&2\int_{0}^{t}\left|\left(\omega_{\tau}\tau^{m\sigma}\nabla^{m}_{\mathcal H_{+}}\Gamma(f, f), \omega_{\tau}\tau^{m\sigma}\nabla^{m}_{\mathcal H_{+}}f\right)_{L^{2}(\mathbb R^{3})}\right|d\tau.
\end{split}
\end{equation*}
For $m=1$, it follows immediately from Proposition \ref{prop2.1}, \eqref{energy-weighted1} that
\begin{align*}
     &\left\|\omega_{t} t^{\sigma}\nabla_{\mathcal H_{+}} f(t)\right\|^2_{L^2(\mathbb R^3)}+\frac{2c_{0}}{(1+T)^{2}}\int_{0}^{t}\left\|\langle\cdot\rangle^{\frac b2}\omega_{\tau}\tau^{\sigma}\nabla_{\mathcal H_{+}} f(\tau)\right\|^{2}_{L^{2}(\mathbb R^{3})}d\tau+\int_{0}^{t}\left\|\omega_{\tau}\tau^{\sigma}\nabla_{\mathcal H_{+}} f(\tau)\right\|^{2}_{A}d\tau\\
     &\le2\sigma\int_{0}^{t}\tau^{2\sigma-1}\left\|\omega_{\tau} \nabla_{\mathcal H_{+}} f(\tau)\right\|^{2}_{L^{2}(\mathbb R^{3})}d\tau+\left(2C_{3}+(4C_{6})^{2}\right)\int_{0}^{t}\left\|\omega_{\tau}\tau^{\sigma}\nabla_{\mathcal H_{+}} f(\tau)\right\|^{2}_{L^{2}(\mathbb R^{3})}d\tau\\
     &\quad+\frac12\left\|\omega_{t} t^{\sigma}\nabla_{\mathcal H_{+}}f\right\|^{2}_{L^{\infty}(]0, c_{0}/2]; L^{2}(\mathbb R^{3}))}+\left((4C_{6})^{2}C_{5}+(4C_{7})^{2}\right)(T+1)^{2\sigma}(\varepsilon_{1})^{2},
\end{align*}
if we choose $C_{4}\varepsilon_{1}\le\frac14$. For the first integral above, note that $0<\frac{-\gamma}{b-\gamma}, \frac{b}{b-\gamma}<1$, by applying H$\rm\ddot o$lder's inequality, one has
\begin{align*}
     \left\|\omega_{t} \nabla_{\mathcal H_{+}} f(\tau)\right\|^{2}_{L^{2}(\mathbb R^{3})}
     &\le\left\|\langle\cdot\rangle^{\frac b2}\omega_{t} \nabla_{\mathcal H_{+}} f\right\|^{\frac{-2\gamma}{b-\gamma}}_{L^{2}(\mathbb R^{3})}\left\|\langle\cdot\rangle^{\frac \gamma2}\omega_{t}\nabla_{\mathcal H_{+}} f\right\|^{\frac{2b}{b-\gamma}}_{L^{2}(\mathbb R^{3})},
\end{align*}
it follows from \eqref{2-2} and \eqref{2-2-1} that
\begin{align*}
     \left\|\langle\cdot\rangle^{\frac \gamma2}\omega_{t} \nabla_{\mathcal H_{+}} f\right\|_{L^{2}(\mathbb R^{3})}&\le\left\|\langle\cdot\rangle^{\frac \gamma2} \nabla_{\mathcal H_{+}}\left(\omega_{t} f\right)\right\|_{L^{2}(\mathbb R^{3})}+\left\|\langle\cdot\rangle^{\frac \gamma2}\left[\omega_{t}, \nabla_{\mathcal H_{+}}\right]f\right\|_{L^{2}(\mathbb R^{3})}\\
     &\le\frac{1}{C_{1}}\left\|\omega_{t} f\right\|_{A}+\frac{bc_{0}}{1+t}\left\|\langle\cdot\rangle^{\frac \gamma2+b-1}\omega_{t} f\right\|_{L^{2}(\mathbb R^{3})}
     \le\frac{bc_{0}+1}{C_{1}}\left\|\omega_{t} f\right\|_{A},
\end{align*}
since $\sigma=\max\{1, \frac{b-\gamma}{2b}\}$, then $\frac{2b}{b-\gamma}\sigma\ge1$, by using H$\rm\ddot o$lder's inequality and \eqref{energy-weighted1}, we have
\begin{align}\label{key-111}
\begin{split}
     &2\sigma\int_{0}^{t}\tau^{2\sigma-1}\left\|\omega_{t} \nabla_{\mathcal H_{+}} f(\tau)\right\|^{2}_{L^{2}(\mathbb R^{3})}d\tau\\
     &\le 2\sigma\left(\frac{bc_{0}+1}{C_{1}}\right)^{\frac{2b}{b-\gamma}}\int_{0}^{t}\tau^{\frac{2b}{b-\gamma}\sigma-1}\left\|\langle\cdot\rangle^{\frac b2}\omega_{\tau} \tau^{\sigma}\nabla_{\mathcal H_{+}} f\right\|^{\frac{-2\gamma}{b-\gamma}}_{L^{2}(\mathbb R^{3})}\left\|\omega_{\tau} f(\tau)\right\|_{A}^{\frac{2b}{b-\gamma}}d\tau\\
     &\le 2\sigma\left(\frac{b(c_{0}+1)^{\sigma+1}}{C_{1}}\right)^{\frac{2b}{b-\gamma}}\left(\int_{0}^{t}\left\|\langle\cdot\rangle^{\frac b2}\omega_{\tau} \tau^{\sigma}\nabla_{\mathcal H_{+}} f\right\|^{2}_{L^{2}(\mathbb R^{3})}d\tau\right)^{\frac{-\gamma}{b-\gamma}}\left(\int_{0}^{t}\left\|\omega_{\tau} f(\tau)\right\|_{A}^{2}d\tau\right)^{\frac{b}{b-\gamma}}\\
     &\le \frac{c_{0}}{(1+T)^{2}}\int_{0}^{t}\left\|\langle\cdot\rangle^{\frac b2}\omega_{\tau} \tau^{\sigma}\nabla_{\mathcal H_{+}} f(\tau)\right\|^{2}_{L^{2}(\mathbb R^{3})}d\tau+\left(\frac{4C_{T}\varepsilon_{0}(b-\gamma)(c_{0}+1)^{2\sigma+1}(T+1)}{C_{1}}\right)^{2}.
\end{split}
\end{align}
From this, we can deduce that for all $0<t\le T$
\begin{align*}
     &\left\|\omega_{t} t^{\sigma}\nabla_{\mathcal H_{+}} f(t)\right\|^2_{L^2(\mathbb R^3)}+\frac{c_{0}}{(1+T)^{2}}\int_{0}^{t}\left\|\langle\cdot\rangle^{\frac b2}\omega_{\tau}\tau^{\sigma}\nabla_{\mathcal H_{+}} f(\tau)\right\|^{2}_{L^{2}(\mathbb R^{3})}d\tau+\int_{0}^{t}\left\|\omega_{\tau}\tau^{\sigma}\nabla_{\mathcal H_{+}} f(\tau)\right\|^{2}_{A}d\tau\\
     &\le2\left(2C_{3}+(4C_{6})^{2}\right)\int_{0}^{t}\left\|\omega_{\tau}\tau^{\sigma}\nabla_{\mathcal H_{+}} f(\tau)\right\|^{2}_{L^{2}(\mathbb R^{3})}d\tau+(\tilde C_{1}\varepsilon_{0})^{2},
\end{align*}
where $\tilde C_{1}$ is a positive constant depends on $c_{0}, \ b, \ \gamma, \ T$ and $C_{1}-C_{7}$. Then from Gronwall's inequality,  we can get that for all $0<t\le T$
\begin{align*}
     &\left\|\omega_{t} t^{\sigma}\nabla_{\mathcal H_{+}} f(t)\right\|^2_{L^2(\mathbb R^3)}+\frac{c_{0}}{(1+T)^{2}}\int_{0}^{t}\left\|\langle\cdot\rangle^{\frac b2}\omega_{\tau}\tau^{\sigma}\nabla_{\mathcal H_{+}} f(\tau)\right\|^{2}_{L^{2}(\mathbb R^{3})}d\tau+\int_{0}^{t}\left\|\omega_{\tau}\tau^{\sigma}\nabla_{\mathcal H_{+}} f(\tau)\right\|^{2}_{A}d\tau\\
     &\le\left(1+2T\left(2C_{3}+(4C_{6})^{2}\right)\right)^{2}e^{2T\left(2C_{3}+(4C_{6})^{2}\right)}(\tilde C_{1}\varepsilon_{0})^{2}.
\end{align*}
Assume $m\ge1$ and \eqref{high-v} holds for all $p \in\mathbb N$ with $1\le p\le m-1$
\begin{equation}\label{hypothesis-v}
\begin{split}
        &\left\|\omega_{t} t^{p\sigma}\nabla_{\mathcal H_{+}}^{p} f(t)\right\|^2_{L^2(\mathbb R^3)}+\frac{c_{0}}{(1+T)^{2}}\int_{0}^{t}\left\|\omega_{\tau} \tau^{p\sigma}\langle\cdot\rangle^{\frac b2}\nabla_{\mathcal H_{+}}^{p} f(\tau)\right\|^2_{L^2(\mathbb R^3)}d\tau\\
        &\quad+\int_0^t\left\|\omega_{\tau}\tau^{p\sigma}\nabla_{\mathcal H_{+}}^{p}f(\tau)\right\|^2_{A}d\tau\le \left(L^{p-1}(p-2)!\right)^{2\sigma},  \quad \forall\ 0<t\le T.
\end{split}
\end{equation}
Now, we shall prove \eqref{hypothesis-v} for $p=m$. We apply the Proposition \ref{prop2.1} to analyse terms $Q_{1}-Q_{3}$. For the term $Q_{1}$, by using \eqref{energy-weighted1}, the Cauchy-Schwarz inequality and hypothesis induction \eqref{hypothesis-v}, one has
\begin{align*}
     Q_{1}&\le(16C_{7}(m-2)!)^{2}(T+1)^{2m\sigma}+\left((\varepsilon_{0}(T+1))^{4\sigma}+\frac18\right)\int_{0}^{t}\left\|\omega_{\tau}\tau^{m\sigma}\nabla_{\mathcal H_{+}}^{m} f(\tau)\right\|^{2}_{A}d\tau\\
     &\quad+128\left(C_{7}\sum_{p=1}^{m-1}\binom{m}{p}\sqrt{\frac{(m-p+2)!}{2}}(T+1)^{(m-p)\sigma}\left(L^{p-1}(p-2)!\right)^{\sigma}\right)^{2}.
\end{align*}
For terms $Q_{2}$ and $Q_{3}$, noting that using \eqref{2-2-1}, one has for all $0<t\le T$
\begin{align*}
     \left\|t^{p\sigma}\omega_{t}\nabla^{p}_{\mathcal H_{+}}f(t)\right\|_{2, \frac\gamma2}&\le\frac{(T+1)^{\sigma}}{C_{1}}\left\|t^{(p-1)\sigma}\omega_{t}\nabla^{p-1}_{\mathcal H_{+}}f(t)\right\|_{A}, \quad \forall \ p\ge1,
\end{align*}
then follows from \eqref{2.8}, \eqref{energy-weighted1}, \eqref{hypothesis-v} and the Cauchy-Schwarz inequality that
\begin{align*}
     Q_{2}
     &\le\frac18\int_{0}^{t}\left\|\tau^{m\sigma}\omega_{\tau}\nabla^{m}_{\mathcal H_{+}}f(\tau)\right\|^{2}_{A}d\tau+T\left(C_{5}\right)^{m}(m+3)!\\
     &\quad+\left(\frac{C_{6}(T+1)^{\sigma}}{C_{1}}\sum_{p=1}^{m}\binom{m}{p}\left(\sqrt{C_{5}}\right)^{m-p}\sqrt{(m-p+3)!}\left(L^{p-2}(p-3)!\right)^{\sigma}\right)^{2},
\end{align*}
from \eqref{2.7}, \eqref{energy-weighted1}, \eqref{hypothesis-v} and the Cauchy-Schwarz inequality, one has for all $0<t\le T$
\begin{align*}
     &Q_{3}\le\left(4C_{4}\varepsilon_{0}+\frac{\varepsilon_{0}C_{4}(T+1)^{\sigma}}{C_{1}}+\frac{1}{16}\right)\int_{0}^{t}\left\|\omega_{\tau} \tau^{m\sigma}\nabla^{m}_{\mathcal H_{+}}f(\tau)\right\|^{2}_{A}d\tau\\
     &\qquad+C_{4}\varepsilon_{1}\left\|\omega_{t} t^{m\sigma}\nabla^{m}_{\mathcal H_{+}}f\right\|^{2}_{L^{\infty}(]0, T]; L^{2}(\mathbb R^{3}))}+\frac{2\varepsilon_{1}C_{4}(T+1)^{\sigma}}{C_{1}}\left(L^{m-2}(m-3)!\right)^{2\sigma}\\
     &\qquad+16\left(\frac{C_{4}(T+1)^{\sigma}}{C_{1}}\sum_{p=2}^{m-1}\binom{m}{p}\left(L^{m-2}(p-2)!(m-p-2)!\right)^{\sigma}\right)^{2}.
\end{align*}
Since $\sigma\ge1$, then from $p!q!\le(p+q)!$, one can deduce
\begin{align}\label{summation}
\begin{split}
     &\sum_{p=2}^{m-1}\binom{m}{p}\left((p-2)!(m-p-2)!\right)^{\sigma}\le4\left((m-2)!\right)^{\sigma}\sum_{p=2}^{m-2}\frac{m^{2}}{p^{2}(m-p)^{2}}\le 64\left((m-2)!\right)^{\sigma},\\
     &\sum_{p=2}^{m-1}\binom{m}{p}\sqrt{\frac{(m-p+2)!}{2}}\left((p-2)!\right)^{\sigma}\le2\sum_{p=2}^{m-1}\binom{m}{p}(m-p-2)!\left((p-2)!\right)^{\sigma}\le128\left((m-2)!\right)^{\sigma},\\
     &\sum_{p=1}^{m}\binom{m}{p}\sqrt{(m-p+3)!}\left((p-3)!\right)^{\sigma}\le6\sum_{p=1}^{m}\binom{m}{p}(m-p-2)!\left((p-2)!\right)^{\sigma}\le400\left((m-2)!\right)^{\sigma},
\end{split}
\end{align}
then taking $L\ge(T+1)^{2\sigma}$ and $\varepsilon_{0}>0$ small enough such that
$$\varepsilon_{0}\le\min\left\{\frac{1}{8(T+1)}, \ \frac{1}{128C_{4}}, \ \frac{C_{1}}{32C_{4}(T+1)^{\sigma}}\right\},$$
we can get that for all $0<t\le T$
\begin{align}\label{integration rest}
\begin{split}
     &\left\|\omega_{t} t^{m\sigma}\nabla_{\mathcal H_{+}}^{m} f(t)\right\|^2_{L^2(\mathbb R^3)}+\frac{2c_{0}}{(1+T)^{2}}\int_{0}^{t}\left\|\langle\cdot\rangle^{\frac b2}\omega_{\tau}\tau^{m\sigma}\nabla_{\mathcal H_{+}}^{m} f(\tau)\right\|^{2}_{L^{2}(\mathbb R^{3})}d\tau+\int_{0}^{t}\left\|\omega_{\tau}\tau^{m\sigma}\nabla_{\mathcal H_{+}}^{m} f(\tau)\right\|^{2}_{A}d\tau\\
     &\le 2m\sigma\int_{0}^{t}\tau^{2m\sigma-1}\left\|\omega_{\tau}\nabla_{\mathcal H_{+}}^{m} f(\tau)\right\|^{2}_{L^{2}(\mathbb R^{3})}d\tau+2C_{3}\int_{0}^{t}\left\|\omega_{\tau}\tau^{m\sigma}\nabla_{\mathcal H_{+}}^{m} f(\tau)\right\|^{2}_{2, \frac\gamma2}d\tau\\
     &\quad+\frac12\left\|\omega_{t} t^{m\sigma}\nabla^{m}_{\mathcal H_{+}}f\right\|^{2}_{L^{\infty}(]0, c_{0}/2]; L^{2}(\mathbb R^{3}))}+B_{0}\left(L^{m-2}(m-2)!\right)^{2\sigma},
\end{split}
\end{align}
with
\begin{align*}
     B_{0}=1000\left(C_{3}+\frac{(C_{4}+C_{6})(T+1)^{\sigma}}{C_{1}}+1\right).
\end{align*}
It remains to consider
$$2m\sigma\int_{0}^{t}\tau^{2m\sigma-1}\left\|\omega_{\tau}\nabla_{\mathcal H_{+}}^{m} f(\tau)\right\|^{2}_{L^{2}(\mathbb R^{3})}d\tau.$$
Applying \eqref{key-111} with $f=\nabla_{\mathcal H_{+}}^{m-1}f$ and the hypothesis \eqref{hypothesis-v}, we have
\begin{align*}
     &2m\sigma\int_{0}^{t}\tau^{2m\sigma-1}\left\|\omega_{\tau} \nabla_{\mathcal H_{+}}^{m} f(\tau)\right\|^{2}_{L^{2}(\mathbb R^{3})}d\tau\\
     &\le \frac{c_{0}}{(1+T)^{2}}\int_{0}^{t}\left\|\langle\cdot\rangle^{\frac b2}\omega_{\tau} \tau^{m\sigma}\nabla_{\mathcal H_{+}}^{m} f\right\|^{2}_{L^{2}(\mathbb R^{3})}d\tau+\left(\frac{4(b-\gamma)(c_{0}+1)(T+1)}{C_{1}}L^{m-2}(m-3)!\right)^{2\sigma}.
\end{align*}
Plugging it back into \eqref{integration rest}, we can get that for all $0<t\le T$
\begin{align}\label{integration rest-1}
\begin{split}
     &\left\|\omega_{t} t^{m\sigma}\nabla_{\mathcal H_{+}}^{m} f(t)\right\|^2_{L^2(\mathbb R^3)}+\frac{c_{0}}{(1+T)^{2}}\int_{0}^{t}\left\|\langle\cdot\rangle^{\frac b2}\omega_{\tau}\tau^{m\sigma}\nabla_{\mathcal H_{+}}^{m} f(\tau)\right\|^{2}_{L^{2}(\mathbb R^{3})}d\tau\\
     &\quad+\int_{0}^{t}\left\|\omega_{\tau}\tau^{m\sigma}\nabla_{\mathcal H_{+}}^{m} f(\tau)\right\|^{2}_{A}d\tau\\
     &\le 4C_{3}\int_{0}^{t}\left\|\omega_{\tau}\tau^{m\sigma}\nabla_{\mathcal H_{+}}^{m} f(\tau)\right\|^{2}_{2, \frac\gamma2}d\tau+\frac12\left\|\omega_{t} t^{m\sigma}\nabla^{m}_{\mathcal H_{+}}f\right\|^{2}_{L^{\infty}(]0, T]; L^{2}(\mathbb R^{3}))}\\
     &\quad+\left(\tilde B_{0}L^{m-2}(m-2)!\right)^{2\sigma},
\end{split}
\end{align}
where
$$\tilde B_{0}=B_{0}+\frac{4(b-\gamma)(c_{0}+1)(T+1)}{C_{1}},$$
this implies that for all $0<t\le T$
\begin{align*}
     &\frac12\left\|\omega_{t} t^{m\sigma}\nabla^{m}_{\mathcal H_{+}}f\right\|^{2}_{L^{\infty}(]0, T]; L^{2}(\mathbb R^{3}))}\le2C_{3}\int_{0}^{t}\left\|\omega_{\tau}\tau^{m\sigma}\nabla_{\mathcal H_{+}}^{m} f(\tau)\right\|^{2}_{L^{2}(\mathbb R^{3})}d\tau+\left(\tilde B_{0}L^{m-2}(m-2)!\right)^{2\sigma},
\end{align*}
substituting it into \eqref{integration rest-1} yields
\begin{align*}
     &\left\|\omega_{t} t^{m\sigma}\nabla_{\mathcal H_{+}}^{m} f(t)\right\|^2_{L^2(\mathbb R^3)}+\frac{c_{0}}{(1+T)^{2}}\int_{0}^{t}\left\|\langle\cdot\rangle\omega_{\tau}\tau^{m\sigma}\nabla_{\mathcal H_{+}}^{m} f(\tau)\right\|^{2}_{L^{2}(\mathbb R^{3})}d\tau+\int_{0}^{t}\left\|\omega_{\tau}\tau^{m\sigma}\nabla_{\mathcal H_{+}}^{m} f(\tau)\right\|^{2}_{A}d\tau\\
     &\le4C_{3}\int_{0}^{t}\left\|\omega_{\tau}\tau^{m\sigma}\nabla_{\mathcal H_{+}}^{m} f(\tau)\right\|^{2}_{L^{2}(\mathbb R^{3})}d\tau+2\left(\tilde B_{0}L^{m-2}(m-2)!\right)^{2\sigma}.
\end{align*}
It follows from Gronwall inequality that for all $0<t\le T$
\begin{align*}
\begin{split}
     &\left\|\omega_{t} t^{m\sigma}\nabla_{\mathcal H_{+}}^{m} f(t)\right\|^2_{L^2(\mathbb R^3)}\le 8C_{3}e^{4TC_{3}}\left(\tilde B_{0}L^{m-2}(m-2)!\right)^{2\sigma},
\end{split}
\end{align*}
therefore, we can deduce that
\begin{align*}
     &\left\|\omega_{t} t^{m\sigma}\nabla_{\mathcal H_{+}}^{m} f(t)\right\|^2_{L^2(\mathbb R^3)}+\frac{c_{0}}{(1+T)^{2}}\int_{0}^{t}\left\|\langle\cdot\rangle^{\frac b2}\omega_{\tau}\tau^{m\sigma}\nabla_{\mathcal H_{+}}^{m} f(\tau)\right\|^{2}_{L^{2}(\mathbb R^{3})}d\tau\\
     &\quad+\int_{0}^{t}\left\|\omega_{\tau}\tau^{m\sigma}\nabla_{\mathcal H_{+}}^{m} f(\tau)\right\|^{2}_{A}d\tau\le\left(L^{m-1}(m-2)!\right)^{2\sigma},  \quad \forall \ t\in]0, T],
\end{align*}
if we choose
$$L\ge\max\left\{4\tilde B_{0}\left(8C_{3}(T+1)+1\right)e^{4TC_{3}}, (T+1)^{2}\right\}.$$
\end{proof}

\bigskip
\noindent {\bf Acknowledgements.} This work was supported by the NSFC (No.12031006) and the Fundamental
Research Funds for the Central Universities of China.

\end{document}